\documentclass[12pt,a4paper]{amsart}
\usepackage{graphicx}
\usepackage{latexsym}
\usepackage{amsxtra}
\usepackage{amsmath}
\usepackage{amsfonts}
\usepackage{amssymb}

\newtheorem{theorem}{Theorem}[section]
\newtheorem{corollary}[theorem]{Corollary}
\newtheorem{lemma}[theorem]{Lemma}
\newtheorem{proposition}[theorem]{Proposition}

\usepackage[figuresright]{rotating}

\usepackage{caption}
\usepackage{subcaption}

\title{Probability distributions with binomial moments}
\author{Wojciech M{\l}otkowski, Karol A. Penson}
\thanks{
W.~M. is supported by the Polish
National Science Center grant No. 2012/05/B/ST1/00626.
K.~A.~P. acknowledges support
from Agence Nationale de la Recherche (Paris, France) under
Program PHYSCOMB No. ANR-08-BLAN-0243-2.}
\address{Instytut Matematyczny,
Uniwersytet Wroc{\l}awski,
Plac~Grunwaldzki~2/4,
50-384 Wroc{\l}aw, Poland}
\email{mlotkow@math.uni.wroc.pl}
\address{Laboratoire de Physique Th\'{e}orique de la Mati\`{e}re
Condens\'{e}e (LPTMC), Universit\'{e} Pierre et Marie Curie, CNRS UMR
7600, Tour 13 - 5i\`{e}me \'{e}t., Bo\^{i}te Courrier 121, 4 place
Jussieu, F 75252 Paris Cedex 05, France}
\email{penson@lptl.jussieu.fr}

\begin{document}

\begin{abstract}
We prove that if $p\ge1$ and $-1\le r\le p-1$ then the binomial sequence
$\binom{np+r}{n}$, $n=0,1,\ldots$, is positive definite and is the moment sequence
of a probability measure $\nu(p,r)$, whose support is contained in
$\left[0,p^p(p-1)^{1-p}\right]$. If $p>1$ is a rational number
and $-1<r\le p-1$ then $\nu(p,r)$ is absolutely continuous
and its density function $V_{p,r}$ can be expressed in terms of the Meijer $G$-function.
In particular cases $V_{p,r}$ is an elementary function.
We show that for $p>1$ the measures $\nu(p,-1)$ and $\nu(p,0)$ are certain free convolution
powers of the Bernoulli distribution.
Finally we prove that the binomial sequence $\binom{np+r}{n}$ is positive definite
if and only if either $p\ge1$, $-1\le r\le p-1$ or $p\le0$, $p-1\le r\le 0$.
The measures corresponding to the latter case are reflections of the
former ones.
\end{abstract}

\maketitle

\section{Introduction}

This paper is devoted to \textit{binomial sequences,} i.e. sequences of the form
\begin{equation}\label{aintbinomial}
\left\{\binom{np+r}{n}\right\}_{n=0}^{\infty},
\end{equation}
where $p,r$ are real parameters. Here the generalized binomial symbol is defined by: $\binom{a}{n}:=a(a-1)\ldots(a-n+1)/n!\,$.
For example, the numbers $\binom{2n}{n}$ are moments of the \textit{arcsine law}
\[
\frac{1}{\pi\sqrt{x(4-x)}}\chi_{(0,4)}(x)\,dx
\]
(see \cite{balnev}).
We are going to prove that if $p\ge1$, $-1\le r\le p-1$, then the sequence (\ref{aintbinomial})
is positive definite and the support of the corresponding probability measure $\nu(p,r)$
is contained in the interval $[0,c(p)]$, where
\begin{equation}\label{aintceodpe}
c(p):=\frac{p^p}{(p-1)^{p-1}}.
\end{equation}
For $r=-1$ this measure has an atom at $x=0$.
If in addition $p>1$ is a rational number, $-1<r\le p-1$,
then $\nu(p,r)$ is absolutely continuous
and the density function $V_{p,r}$ can be expressed in terms
of the Meijer $G$- (and consequently of the generalized hypergeometric) functions.
In particular cases $V_{p,r}$ is an elementary function.

Similar problems were studied in \cite{mlotkowski2010,mpz2012} for \textit{Raney sequences}
\begin{equation}\label{aintraney}
\left\{\binom{np+r}{n}\frac{r}{np+r}\right\}_{n=0}^{\infty}
\end{equation}
(called \textit{Fuss sequences} if $r=1$).
It was shown that if $p\ge1$ and $0\le r\le p$ then the Raney sequence (\ref{aintraney})
is positive definite and the corresponding probability measure $\mu(p,r)$ has compact support
contained in $[0,c(p)]$.
In particular $\mu(2,1)$ is the Marchenko-Pastur distribution,
which plays an important role in the theory of random matrices,
see \cite{vdn,ns,pezy}.
Moreover, for $p>0$ we have $\mu(p,1)=\mu(2,1)^{\boxtimes p-1}$,
where ``$\boxtimes$" denotes the multiplicative free convolution.

The paper is organized as follows.
First we study the generating function $\mathcal{D}_{p,r}$
of the sequence (\ref{aintbinomial}). For particular cases,
namely for $p=2,3,3/2$, we express $\mathcal{D}_{p,r}$ as an elementary function.
In the next part we prove that if $p\ge1$ and $-1\le r\le p-1$
then the sequence is positive definite
and the support of the corresponding probability measure $\nu(p,r)$
is contained in $[0,c(p)]$.
If $p>1$ is rational and $-1<r\le p-1$  then
$\nu(p,r)$ can be expressed as the Mellin convolution
of modified beta measures, in particular
$\nu(p,r)$ is absolutely continuous, while $\nu(p,-1)$ has an atomic part at $0$.
Note that the positive definiteness of the binomial sequence (\ref{aintbinomial})
was already proved in \cite{mlotkowski2010} under more
restrictive assumptions (namely, that $0\le r\le p-1$)
and the proof involved the multiplicative free and the monotonic convolution.

In the next section we study the density function
$V_{p,r}$ of the absolutely continuous measures
$\nu(p,r)$, where $p>1$ is rational, $-1<r\le p-1$.
We show that $V_{p,r}$ can be expressed as the Meijer
$G$-function, and therefore as linear combination of the
generalized hypergeometric functions.
In particular we derive an elementary formula for $p=2$.

Then we concentrate on the cases $p=3$ and $p=3/2$.
For particular choices of $r$
(namely, $r=0,1,2$ for $p=3$ and $r=-1/2,0,1/2$ for $r=3/2$)
we express $V_{p,r}$ as an elementary function
(Theorem~\ref{etwtrzy} and Theorem~\ref{etwtrzydrugie}).

Some of the sequences (\ref{aintbinomial}) have combinatorial
applications and appear in the Online Encyclopedia of Integer Sequences \cite{oeis}~(OEIS).
Perhaps the most important is $\binom{2n}{n}$ (A000984 in OEIS),
the moment sequence of the arcsine distribution.
The sequences $\binom{3n-1}{n}$, $\binom{3n}{n}$, $\binom{3n+1}{n}$
and $\binom{3n+2}{n}$ can be found in OEIS as A165817, A005809, A045721 and A025174
respectively. We also shed some light on sequence~A091527:~$\binom{3n/2-1/2}{n}4^n$,
as well as on A061162, the even numbered terms of the former
(see remarks following Theorem~\ref{etwtrzydrugie}).

In Section~6 we study various convolution relations involving the measures
$\nu(p,r)$ and $\mu(p,r)$. For example we show in Proposition~\ref{gpropbernoulli}
that the measures $\nu(p,-1)$ and $\nu(p,0)$ are certain free convolution
powers of the Bernoulli distribution.

In Section~7 we prove that the binomial sequence (\ref{aintbinomial}) is positive definite
if and only if either $p\ge1$, $-1\le r\le p-1$ or $p\le0$, $p-1\le r\le 0$.
The measures corresponding to the latter case are reflections of those
corresponding to the former one. Similarly, the Raney sequence (\ref{aintraney})
is positive definite if and only if either
$p\ge1$, $0\le r\le p$ or $p\le0$, $p-1\le r\le0$ or else $r=0$
(the case $r=0$ the corresponds to $\delta_{0}$).
Quite surprisingly, the proof involves the monotonic convolution
introduced by Muraki~\cite{muraki}.

Finally, we provide graphical representation for selected functions $V_{p,r}$.

\section{Generating functions}

In this part we are going to study the generating function
\begin{equation}\label{bgendgenfunct}
\mathcal{D}_{p,r}(z):=\sum_{n=0}^{\infty}\binom{np+r}{n}z^n
\end{equation}
(convergent in some neighborhood of $0$)
of the binomial sequence (\ref{aintbinomial}).
First we observe relations between the functions
$\mathcal{D}_{p,-1}$, $\mathcal{D}_{p,0}$ and $\mathcal{D}_{p,p-1}$.

\begin{proposition}\label{bgenpropddd}
For every $p\in\mathbb{R}$ we have
\begin{equation}\label{bgendeminusjeden}
\mathcal{D}_{p,-1}(z)=\frac{1}{p}+\frac{p-1}{p}\mathcal{D}_{p,0}(z)
\end{equation}
and
\begin{equation}\label{bgendepeminusjeden}
\mathcal{D}_{p,p-1}(z)=\frac{\mathcal{D}_{p,0}(z)-1}{pz}.
\end{equation}
\end{proposition}

\begin{proof}
These formulas are consequences of the following elementary identities:
\begin{equation}\label{bgenbinomialrelation}
\frac{1}{p-1}\binom{(n+1)p-1}{n+1}=\frac{1}{p}\binom{(n+1)p}{n+1}=\binom{np+p-1}{n},
\end{equation}
valid for $p\in\mathbb{R}$, $n=0,1,2,\ldots$.
\end{proof}

It turns out that $\mathcal{D}_{p,r}$ is related
to the generating function
\begin{equation}\label{bgenbgenfunct}
\mathcal{B}_{p}(z):=\sum_{n=0}^{\infty}\binom{np+1}{n}\frac{z^n}{np+1}
\end{equation}
of the Fuss numbers.
This function satisfies equation
\begin{equation}
\mathcal{B}_{p}(z)=1+z\cdot \mathcal{B}_{p}(z)^{p},
\end{equation}
with the initial value $\mathcal{B}_{p}(0)=1$ (5.59 in \cite{gkp}),
and Lambert's formula
\begin{equation}
\mathcal{B}_{p}(z)^{r}=\sum_{n=0}^{\infty}
\binom{np+r}{n}\frac{r\cdot z^n}{np+r}.
\end{equation}

Since
\begin{equation}\label{bgenbdgenfunct}
\mathcal{D}_{p,r}(z)=\frac{\mathcal{B}_{p}(z)^{1+r}}{p-(p-1)\mathcal{B}_{p}(z)}
\end{equation}
(5.61 in \cite{gkp}), it is sufficient to study the functions $\mathcal{B}_{p}$.

The simplest cases for $\mathcal{B}_{p}$ are:
\begin{align*}
\mathcal{B}_{0}(z)&=1+z,\\
\mathcal{B}_{1}(z)&=\frac{1}{1-z},\\
\mathcal{B}_{-1}(z)&=\frac{1+\sqrt{1+4z}}{2},\\
\mathcal{B}_{2}(z)&=\frac{2}{1+\sqrt{1-4z}},\\
\mathcal{B}_{1/2}(z)&=\frac{2+z^2+z\sqrt{4+z^2}}{2},
\end{align*}
which lead to
\begin{align*}
\mathcal{D}_{0,r}(z)&=(1+z)^{r},\\
\mathcal{D}_{1,r}(z)&=(1-z)^{-1-r},\\
\mathcal{D}_{-1,r}(z)&=\frac{\left(\frac{1+\sqrt{1+4z}}{2}\right)^{1+r}}{\sqrt{1+4z}},\\
\mathcal{D}_{2,r}(z)&=\frac{\left(\frac{2}{1+\sqrt{1-4z}}\right)^r}{\sqrt{1-4z}},\\
\mathcal{D}_{1/2,r}(z)&=\frac{4\left(\frac{2+z^2+z\sqrt{4+z^2}}{2}\right)^{1+r}}{4+z^2+z\sqrt{4+z^2}}.
\end{align*}
These examples illustrate the following general rules:
\begin{align}
\mathcal{B}_{p}(z)&=\mathcal{B}_{1-p}(-z)^{-1},\label{bbreflection}\\
\mathcal{D}_{p,r}(z)&=\mathcal{D}_{1-p,-1-r}(-z).\label{bdreflection}
\end{align}

The rest of this section is devoted to the cases $p=3$ and $p=3/2$.

\subsection{The case $p=3$}

First we find $\mathcal{B}_{3}$.

\begin{proposition}\label{bgenpropb3}
For $|z|<4/27$ we have
\begin{equation}
\mathcal{B}_3(z)=\frac{3}{3\cos^2\alpha-\sin^2\alpha},
\end{equation}
where $\alpha=\frac{1}{3}\arcsin\left(\sqrt{27z/4}\right)$.
\end{proposition}

Note that both the maps $u\mapsto \cos^2\left(\frac{1}{3}\arcsin(u)\right)$
and  $u\mapsto \sin^2\left(\frac{1}{3}\arcsin(u)\right)$ are even,
hence involve only even powers of $u$ in their Taylor expansion.
Therefore the functions $u\mapsto \cos^2\left(\frac{1}{3}\arcsin(\sqrt{u})\right)$
and $u\mapsto \sin^2\left(\frac{1}{3}\arcsin(\sqrt{u})\right)$
are well defined and analytic on the disc $|u|<1$.

\begin{proof}
First we note that
\[
\binom{3n+1}{n}\frac{1}{3n+1}=\frac{(3n)!}{(2n+1)!n!}
=\frac{\left(\frac{1}{3}\right)_n\left(\frac{2}{3}\right)_n\left(\frac{3}{3}\right)_n 3^{3n}}
{\left(\frac{2}{2}\right)_n\left(\frac{3}{2}\right)_n 2^{2n}\cdot n!},
\]
where $(a)_n:=a(a+1)\ldots(a+n-1)$ is the \textit{Pochhammer symbol}.
This implies that
\[
\mathcal{B}_3(z)=
{}_{2}F_{1}\!\left(\left.\frac{1}{3},\frac{2}{3};\,\frac{3}{2}\right|\frac{27z}{4}\right).
\]
Now, applying the identity
\[{}_{2}F_{1}\!\left(\left.a,1-a;\,3/2\right|\sin^2 t\right)=\frac{\sin\left((2a-1)t\right)}{(2a-1)\sin t},
\]
(see 15.4.14 in \cite{olver}) with $t=3\alpha$ and $a=2/3$, we get
\[
\mathcal{B}_3(z)=
\frac{3\sin\alpha}{\sin 3\alpha}=\frac{3}{3\cos^2\alpha-\sin^2\alpha}.
\]
\end{proof}

Now we can give formula for $\mathcal{D}_{3,r}$.

\begin{corollary}\label{bgencord3}
For all $r\in\mathbb{R}$ we have
\begin{equation}
\mathcal{D}_{3,r}(z)
=\left(\cfrac{3}{3\cos^2\alpha-\sin^2\alpha}\right)^{r}\frac{1}{\cos^2\alpha-3\sin^2\alpha},
\end{equation}
where $\alpha=\frac{1}{3}\arcsin\left(\sqrt{27z/4}\right)$, $|z|<27/4$.
\end{corollary}

Now we observe that $\mathcal{D}_{3,r}(z)$ and $\mathcal{B}_{3}(z)^r$
can be also expressed as hypergeometric functions.

\begin{proposition}
For all $r\in\mathbb{R}$ and $|z|<4/27$ we have
\begin{align}
\mathcal{D}_{3,r}(z)
&={}_{3}F_{2}\!\left(\left.\frac{1+r}{3},\frac{2+r}{3},\frac{3+r}{3};\,
\frac{1+r}{2},\frac{2+r}{2}\right|\frac{27z}{4}\right),\\
\mathcal{B}_{3}(z)^r
&={}_{3}F_{2}\!\left(\left.\frac{r}{3},\frac{1+r}{3},\frac{2+r}{3};\,
\frac{1+r}{2},\frac{2+r}{2}\right|\frac{27z}{4}\right).
\end{align}
\end{proposition}

\begin{proof} It is easy to check that
\[
\binom{3n+r}{n}=\frac{\left(\frac{1+r}{3}\right)_n\left(\frac{2+r}{3}\right)_n\left(\frac{3+r}{3}\right)_n 3^{3n}}
{\left(\frac{1+r}{2}\right)_n\left(\frac{2+r}{2}\right)_n 2^{2n}\cdot n!}
\]
and
\[
\binom{3n+r}{n}\frac{r}{3n+r}=\frac{\left(\frac{r}{3}\right)_n\left(\frac{1+r}{3}\right)_n\left(\frac{2+r}{3}\right)_n 3^{3n}}
{\left(\frac{1+r}{2}\right)_n\left(\frac{2+r}{2}\right)_n 2^{2n}\cdot n!},
\]
which leads to the statement.
\end{proof}

As a byproduct we obtain two hypergeometric identities:

\begin{corollary}
For $a\in\mathbb{R}$, $|u|<1$ we have
\begin{align}
{}_{3}F_{2}\!\left(\left.a,a+\frac{1}{3},a+\frac{2}{3};\,\frac{3a}{2},\frac{3a+1}{2}\right|u\right)
&=\frac{\left(\cfrac{3}{3\cos^2 \alpha-\sin^2 \alpha}\right)^{3a-1}}{\cos^2 \alpha-3\sin^2\alpha},
\label{hyperformula1}\\
{}_{3}F_{2}\!\left(\left.a,a+\frac{1}{3},a+\frac{2}{3};\,\frac{3a+1}{2},\frac{3a+2}{2}\right|u\right)
&=\left(\cfrac{3}{3\cos^2 \alpha-\sin^2 \alpha}\right)^{3a},\label{hyperformula2}
\end{align}
where $\alpha=\frac{1}{3}\arcsin\sqrt{u}$.
\end{corollary}

\textbf{Remark.}
One can check, that (\ref{hyperformula1}) and (\ref{hyperformula2}) are alternative versions of the following known formulas:
\begin{align}
{}_{3}F_{2}\!\left(\left.a,a+\frac{1}{3},a+\frac{2}{3};\,\frac{3a}{2},\frac{3a+1}{2}\right|\frac{-27z}{4(1-z)^3}\right)
&=\frac{(1-z)^{3a}}{2z+1},\\
{}_{3}F_{2}\!\left(\left.a,a+\frac{1}{3},a+\frac{2}{3};\,\frac{3a+1}{2},\frac{3a+2}{2}\right|\frac{-27z}{4(1-z)^3}\right)
&=(1-z)^{3a}
\end{align}
(7.4.1.28 and 7.4.1.29 in \cite{prudnikov3}).
Indeed, putting
\[
z=\frac{-4\sin^2\alpha}{3\cos^2\alpha-\sin^2\alpha}
\]
we have
\[
1-z=\frac{3}{3\cos^2\alpha-\sin^2\alpha},
\]
\[
\frac{-27z}{4(1-z)^3}=\sin^2\alpha\left(3\cos^2\alpha-\sin^2\alpha\right)^2=\sin^2\ 3\alpha
\]
and
\[
2z+1=\frac{3\left(\cos^2\alpha-3\sin^2\alpha\right)}{3\cos^2\alpha-\sin^2\alpha}.
\]

Let us mention here that the sequences $\binom{3n-1}{n}$, $\binom{3n}{n}$, $\binom{3n+1}{n}$
and $\binom{3n+2}{n}$ appear in OEIS as A165817, A005809, A045721 and A025174
respectively.

\subsection{The case $p=3/2$}

First we compute $\mathcal{B}_{3/2}$ in terms of hypergeometric functions.

\begin{lemma}\label{bgenlemmab32}
\[
\mathcal{B}_{3/2}(z)=
\frac{1-{}_{2}F_{1}\!\left(\left.\frac{-2}{3},\frac{-1}{3};\,\frac{-1}{2}\right|\frac{27z^2}{4}\right)}{3z^2}
+z\cdot{}_{2}F_{1}\!\left(\left.\frac{5}{6},\frac{7}{6};\,\frac{5}{2}\right|\frac{27z^2}{4}\right).
\]
\end{lemma}

\begin{proof} If $n=2k$ then the coefficient at $z^n$ on the right hand side is
\[
-\frac{\left(\frac{-2}{3}\right)_{k+1}\left(\frac{-1}{3}\right)_{k+1}3^{3k+3}}
{3\left(\frac{-1}{2}\right)_{k+1}(k+1)!2^{2k+2}}
\]
\[
=\frac{-(-2)\cdot 1\cdot 4\cdot\ldots\cdot(3k-2)\cdot(-1)\cdot2\cdot5\cdot\ldots(3k-1)3^{k+1}}
{3(-1)\cdot 1\cdot 3\cdot\ldots\cdot(2k-1)(k+1)!2^{k+1}}
\]
\[
=\frac{1\cdot 4\cdot\ldots\cdot(3k-2)\cdot2\cdot5\cdot\ldots(3k-1)3^{k}}
{1\cdot 3\cdot\ldots\cdot(2k-1)(k+1)!2^{k}}
\]
\[
=\frac{(3k)!}{(2k)!(k+1)!}=\binom{3k+1}{2k}\frac{1}{3k+1}=\binom{3n/2+1}{n}\frac{1}{3n/2+1}.
\]

Now assume that $n=2k+1$. Then
\[
\binom{3n/2+1}{n}\frac{1}{3n/2+1}=\frac{(6k+3)(6k+1)(6k-1)\ldots(2k+5)}{2^{2k}(2k+1)!}
=\frac{(6k+3)!!}{2^{2k}(2k+3)!!(2k+1)!}.
\]
On the other hand
\[
\frac{\left(\frac{5}{6}\right)_{k}\left(\frac{7}{6}\right)_{k}3^{3k}}{\left(\frac{5}{2}\right)_{k}2^{2k}k!}
=\frac{5\cdot11\cdot\ldots\cdot(6k-1)\cdot 7\cdot13\cdot\ldots\cdot(6k+1)\cdot 3^k}{5\cdot 7\cdot\ldots(2k+3)2^{2k}k! 2^k}
\]
\[
=\frac{(6k+3)!!}{2^{2k}(2k+3)!!(2k+1)!!2^k k!}.
\]
Since $(2k+1)!! 2^k k!=(2k+1)!\,$, the proof is completed.
\end{proof}

Now we find formulas for these two hypergeometric functions.

\begin{lemma}
\begin{align}
{}_{2}F_{1}\!\left(\left.\frac{-2}{3},\frac{-1}{3};\,\frac{-1}{2}\right|u\right)
&=\frac{2}{3}\cos 2\beta+\frac{1}{3}\cos4\beta,\label{bhyperg1}\\
{}_{2}F_{1}\!\left(\left.\frac{5}{6},\frac{7}{6};\,\frac{5}{2}\right|u\right)
&=\frac{27\cos\beta\sin^3\beta}{\sin^3 3\beta},\label{bhyperg2}
\end{align}
where $\beta=\frac{1}{3}\arcsin\left(\sqrt{u}\right)$.
\end{lemma}

\begin{proof}
We know that ${}_{2}F_{1}\!\left(\left.a,b;\,c\right|z\right)$
is the unique function $f$ which is analytic at $z=0$,
with $f(0)=1$, and satisfies the \textit{hypergeometric equation}:
\[
z(1-z)f''(z)+\big[c-(a+b+1)z\big]f'(z)-abf(z)=0
\]
(see \cite{andrews}).
Now one can check that this equation is satisfied
by the right hand sides of these equations
(\ref{bhyperg1}) and (\ref{bhyperg2})
for given parameters~$a,b,c$.
\end{proof}

Now we are ready to express $\mathcal{B}_{3/2}$ as an elementary function.

\begin{proposition}
For $3|z|\sqrt{3}<2$ we have
\[
\mathcal{B}_{3/2}(z)
=\frac{3}{\left(\sqrt{3}\cos\beta-\sin\beta\right)^2}
\]
where $\beta=\frac{1}{3}\arcsin\left(3z\sqrt{3}/2\right)$.
\end{proposition}

\begin{proof}
In view of the previous lemmas we have
\[
\mathcal{B}_{3/2}(z)=\frac{1-\frac{2}{3}\cos2\beta-\frac{1}{3}\cos4\beta}{\frac{4}{9}\sin^2 3\beta}
+\frac{2\sin 3\beta}{3\sqrt{3}}\frac{27\cos\beta\sin^3\beta}{\sin^3 3\beta}
\]
\[
=\frac{3(1-\cos2\beta)(2+\cos2\beta)}{2\sin^2 3\beta}
+\frac{6\sqrt{3}\cos\beta\sin^3\beta}{\sin^2 3\beta}
\]
\[
=\frac{3\sin^2 \beta\left(3\cos^2\beta+\sin^2\beta+2\sqrt{3}\cos\beta\sin\beta\right)}{\sin^2 3\beta}
\]
\[
=\frac{3\left(\sqrt{3}\cos\beta+\sin\beta\right)^2}{\left(3\cos^2\beta-\sin^2\beta\right)^2}
=\frac{3}{\left(\sqrt{3}\cos\beta-\sin\beta\right)^2}.
\]
\end{proof}

Now we provide formula for $\mathcal{D}_{3/2,r}$.

\begin{corollary}\label{bcortrzydrugie}
For all $r\in\mathbb{R}$ and $3|z|\sqrt{3}<2$ we have
\[
\mathcal{D}_{3/2,r}(z)=\left(\cfrac{3}{\left(\sqrt{3}\cos\beta-\sin\beta\right)^2}\right)^r
\frac{1}{\cos\beta\left(\cos\beta-\sqrt{3}\sin\beta\right)},
\]
where $\beta=\frac{1}{3}\arcsin\left(3z\sqrt{3}/2\right)$.
\end{corollary}

Note also a hypergeometric expression for $\mathcal{D}_{3/2,r}$:

\begin{proposition}
For all $r\in\mathbb{R}$ and $3|z|\sqrt{3}<2$ we have
\[
\mathcal{D}_{3/2,r}(z)
={}_{3}F_{2}\!\left(\left.\frac{1+r}{3},\frac{2+r}{3},\frac{3+r}{3};\,
\frac{1}{2},1+r\right|\frac{27z^2}{4}\right)
\]
\[
+\frac{z(2r+3)}{2}{}_{3}F_{2}\!\left(\left.\frac{5+2r}{6},\frac{7+2r}{6},\frac{9+2r}{6};\,
\frac{3}{2},\frac{3+2r}{2}\right|\frac{27z^2}{4}\right).
\]
\end{proposition}

\begin{proof}
If $n=2k$ then the coefficient at $z^n$ on the right hand side is
\[
\frac{\left(\frac{1+r}{3}\right)_{k}\left(\frac{2+r}{3}\right)_{k}\left(\frac{3+r}{3}\right)_{k}3^{3k}}
{\left(\frac{1}{2}\right)_{k}\left(1+r\right)_{k}2^{2k}k!}
=\frac{(1+r)(2+r)(3+r)\ldots(3k+r)}{(2k-1)!!(1+r)\ldots(k+r)2^k k!}
\]
\[
=\frac{(k+1+r)(k+2+r)\ldots(3k+r)}{(2k)!}=\binom{3k+r}{2k}.
\]
If, in turn, $n=2k+1$ then the coefficient at $z^n$ is
\[
\frac{3+2r}{2}\frac{\left(\frac{5+2r}{6}\right)_{k}\left(\frac{7+2r}{6}\right)_{k}\left(\frac{9+2r}{6}\right)_{k}3^{3k}}
{\left(\frac{3}{2}\right)_{k}\left(\frac{3+2r}{2}\right)_{k}2^{2k}k!}\]
\[
=\frac{3+2r}{2}\frac{(5+2r)(7+2r)(9+2r)(11+2r)\ldots(6k+3+2r)}
{(2k+1)!!(3+2r)(5+2r)\ldots(2k+1+2r)2^{3k}k!}
\]
\[
=\frac{(2k+3+2r)(2k+5+2r)\ldots(6k+3+2r)}{(2k+1)! 2^{2k+1}}\]
\[
=\frac{\left(\frac{3(2k+1)}{2}+r\right)\left(\frac{3(2k+1)}{2}-1+r\right)\left(\frac{3(2k+1)}{2}-2+r\right)
\ldots\left(\frac{3(2k+1)}{2}-2k+r\right)}{(2k+1)!}
\]
\[
=\binom{3(2k+1)/2+r}{2k+1},
\]
which proves the odd case.
\end{proof}

\section{Mellin convolution}

In this part we are going to prove that if $p\ge1$ and $-1\le r\le p-1$
then the sequence (\ref{aintbinomial}) is positive definite.
Moreover, if $p>1$ is rational and $-1<r\le p-1$
then the corresponding probability measure $\nu(p,r)$
is absolutely continuous and is the Mellin product
of modified beta distributions, see~\cite{balnev}.

\begin{lemma}\label{cmellemma1}
If $p=k/l$, where $k,l$ are integers, $1\le l<k$, $r>-1$
and if $mp+r+1\ne0,-1,-2,\ldots$ then
\begin{equation}\label{cmelbinom}
\binom{mp+r}{m}
=\frac{1}{\sqrt{2\pi l}}\left(\frac{p}{p-1}\right)^{r+1/2}
\frac{\prod_{j=1}^{k}\Gamma(\beta_j+m/l)}{\prod_{j=1}^{k}\Gamma(\alpha_j+m/l)}
c(p)^m,
\end{equation}
$m=0,1,2,\ldots$, where $c(p)=p^p(p-1)^{1-p}$,
\begin{align}
\alpha_j&=\left\{\begin{array}{ll}
\cfrac{j}{l}&\mbox{if $1\le j\le l$,}\label{cmelalpha}\\
\cfrac{r+j-l}{k-l}&\mbox{if $l+1\le j\le k$,}
\end{array}\right.\\
\beta_j&=\frac{r+j}{k},\quad\qquad{1\le j\le k}.\label{cmelbetha}
\end{align}
\end{lemma}

Writing $p=k/l$ we will tacitly assume that $k,l$ are relatively prime,
although this assumption is not necessary in the sequel.

\begin{proof}
Assuming that  $mp+r+1\ne0,-1,-2,\ldots$, we have
\begin{equation}\label{cmelbinomgamma}
\binom{mp+r}{m}=\frac{\Gamma(mp+r+1)}{\Gamma(m+1)\Gamma(mp-m+r+1)}.
\end{equation}
Now we apply the \textit{Gauss's multiplication formula}:
\begin{equation}
\Gamma(nz)=(2\pi)^{(1-n)/2}n^{nz-1/2}
\prod_{i=0}^{n-1}\Gamma\left(z+\frac{i}{n}\right)
\end{equation}
which gives us:
\begin{align*}
\Gamma(m p+r+1)&=\Gamma\left(k\left(\frac{m}{l}+\frac{r+1}{k}\right)\right)
=(2\pi)^{(1-k)/2}k^{m k/l+r+1/2}\prod_{j=1}^{k}
\Gamma\left(\frac{m}{l}+\frac{r+j}{k}\right),\\
\Gamma(m+1)&=\Gamma\left(l\frac{m+1}{l}\right)=(2\pi)^{(1-l)/2}l^{m+1/2}\prod_{j=1}^{l}
\Gamma\left(\frac{m}{l}+\frac{j}{l}\right)
\end{align*}
and
\begin{align*}
\Gamma(m p-m+r+1)&=\Gamma\left((k-l)\left(\frac{m}{l}+\frac{r+1}{k-l}\right)\right)\\
&=(2\pi)^{(1-k+l)/2}(k-l)^{m(k-l)/l+r+1/2}\prod_{j=l+1}^{k}
\Gamma\left(\frac{m}{l}+\frac{r+j-l}{k-l}\right).
\end{align*}
Applying to (\ref{cmelbinomgamma}) we get (\ref{cmelbinom}).
\end{proof}

Similarly as in \cite{mpz2012} we need to change the enumeration of $\alpha$'s.
Note that here this modification depends not only on $k,l$ but also on $r$.

\begin{lemma}\label{cmellemma2}
Suppose that $k,l$ are integers such that $1\le l<k$
and that $-1< r\le p-1=(k-l)/l$. For $1\le i\le l$ define
\[
j_i':=\left\lfloor\frac{ik}{l}-r\right\rfloor,
\]
where $\lfloor\cdot\rfloor$ denotes the floor function.
In addition we put $j_0':=0$ and $j_{l+1}':=k+1$, so that
\[0=j_0'<j_1'<j_2'<\ldots<j_l'\le k<k+1=j_{l+1}'.\]
For $1\le j\le k$ define
\begin{equation}
\widetilde{\alpha}_j=\left\{\begin{array}{ll}
\cfrac{i}{l}&\mbox{ if $j=j_i'$, $1\le i\le l$,}\\
\cfrac{r+j-i}{k-l}&\mbox{ if $j_i'<j<j_{i+1}'$.}
\end{array}\right.
\end{equation}
Then the sequence $\left\{\widetilde{\alpha}_j\right\}_{j=1}^{k}$
is a permutation of $\left\{{\alpha}_j\right\}_{j=1}^{k}$
and we have $\beta_j\le\widetilde{\alpha}_j$
for all $j\le k$.
\end{lemma}

\begin{proof}
If $j=j_i'$, $1\le i\le l$, then we have to prove that
\[
\frac{r+j_i'}{k}\le \frac{i}{l},
\]
which is equivalent to
\[
j_i'\le\frac{ik}{l}-r,
\]
but this is a consequence of the definition of $j_i'$ and the inequality $\lfloor x\rfloor\le x$.

Now assume that $j_i'<j<j_{i+1}'$, $0\le i\le k$. Then we should prove that
\[
\frac{r+j}{k}\le \frac{r+j-i}{k-l},
\]
which is equivalent to
\[
lr+lj-ik\ge0.
\]
Since $\lfloor x\rfloor+1>x$, we have
\[
lr+lj-ik\ge lr+l(j_i'+1)-ik>lr+l\left(\frac{ik}{l}-r\right)-ik=0,
\]
which concludes the proof.
\end{proof}

Recall that for probability measures $\mu_1$, $\mu_2$
on the positive half-line $[0,\infty)$ the \textit{Mellin convolution} (or the \textit{Mellin product}) is defined by
\begin{equation}
\left(\mu_1\circ\mu_2\right)(A):=\int_{0}^{\infty}\int_{0}^{\infty}\mathbf{1}_{A}(xy)d\mu_1(x)d\mu_{2}(y)
\end{equation}
for every Borel set $A\subseteq[0,\infty)$.
This is the distribution of the product $X_1\cdot X_2$ of two independent nonnegative
random variables with $X_i\sim\mu_i$.
In particular,
$\mu\circ\delta_c$ is the \textit{dilation} of $\mu$:
\[
\left(\mu\circ\delta_c\right)(A)=\mathbf{D}_c\mu(A):=\mu\left(\frac{1}{c}A\right)
\]
($c>0$). If $\mu$ has density $f(x)$
then $\mathbf{D}_c(\mu)$ has density $f(x/c)/c$.

If both the measures $\mu_1,\mu_2$ have all \textit{moments}
\[
s_m(\mu_i):=\int_{0}^{\infty}x^m\,d\mu_i(x)
\]
finite then so has $\mu_1\circ\mu_2$ and
\[
s_m\left(\mu_1\circ\mu_2\right)=s_m(\mu_1)\cdot s_m(\mu_2)
\]
for all $m$.

If $\mu_1,\mu_2$ are absolutely continuous, with densities $f_1,f_2$ respectively, then
so is $\mu_1\circ\mu_2$ and its density is given by the Mellin convolution:
\[
\left(f_1\circ f_2\right)(x):=\int_{0}^{\infty}f_{1}(x/y)f_{2}(y)\frac{dy}{y}.
\]

Similarly as in \cite{mpz2012} we will use
the \textit{modified beta distributions} (see~\cite{balnev}):
\begin{equation}\label{modifiedbeta}
\mathbf{b}(u+v,u,l):=
\frac{l}{\mathrm{B}(u,v)}x^{lu-1}\left(1-x^l\right)^{v-1}dx,
\qquad x\in[0,1],
\end{equation}
where $u,v,l>0$ and $\mathrm{B}$ denotes the Euler beta function.
The $n$th moment of $\mathbf{b}(u+v,u,l)$ is
\[
\int x^n\,d\mathbf{b}(u+v,u,l)(x)=\frac{\Gamma(u+n/l)\Gamma(u+v)}{\Gamma(u+v+n/l)\Gamma(u)}.
\]
We also define $\mathbf{b}(u,u,l):=\delta_{1}$
for $u,l>0$.

Now we are ready to prove

\begin{theorem}\label{cmelthmellin}
Suppose that $p=k/l$, where $k,l$ are integers such that
$1\le l<k$, and that $r$ is a real number, $-1<r\le p-1$.
Then there exists a unique probability measure $\nu(p,r)$
such that $\binom{mp+r}{m}$ is its moment sequence.
Moreover, $\nu(p,r)$ can be represented as the following Mellin
convolution:
\[
\nu(p,r)=\mathbf{b}(\widetilde{\alpha}_1,\beta_1,l)\circ\ldots\circ
\mathbf{b}(\widetilde{\alpha}_k,\beta_k,l)\circ\delta_{c(p)},
\]
where $c(p):=p^p(p-1)^{1-p}$.
In particular, $\nu(p,r)$ is absolutely continuous
and its support is $[0,c(p)]$.
\end{theorem}

The density  function of $\nu(p,r)$ will be denoted by $V_{p,r}(x)$.

\begin{proof}
In view of Lemma~\ref{cmellemma1} and Lemma~\ref{cmellemma2}
we can write
\[
\binom{mp+r}{m}=D
\prod_{j=1}^{k}\frac{\Gamma(\beta_j+m/l)\Gamma(\widetilde{\alpha}_j)}
{\Gamma(\widetilde{\alpha}_j+m/l)\Gamma(\beta_j)}\cdot c(p)^m
\]
for some constant $D$. Taking $m=0$ we see that $D=1$.
\end{proof}

\textbf{Example.} Assume that $p=2$. If $-1<r\le0$ then
\begin{equation}
\nu(2,r)=\mathbf{b}\left(r+1,\frac{r+1}{2},1\right)\circ
\mathbf{b}\left(1,\frac{r+2}{2},1\right)\circ\delta_{4},
\end{equation}
and if $0\le r\le1$ then
\begin{equation}
\nu(2,r)=\mathbf{b}\left(1,\frac{r+1}{2},1\right)\circ
\mathbf{b}\left(r+1,\frac{r+2}{2},1\right)\circ\delta_{4}.
\end{equation}

\begin{theorem}\label{cthepositive}
Suppose that $p,r$ are real numbers, $p\ge1$ and $-1\le r\le p-1$.
Then there exists a unique probability measure $\nu(p,r)$,
with support contained in $[0,c(p)]$, such that
$\left\{\binom{mp+r}{m}\right\}_{m=0}^{\infty}$ is its moment sequence.
\end{theorem}

\begin{proof}
It follows from the fact that the class of positive definite sequence
is closed under pointwise limits.
\end{proof}

Recall that if $\left\{s_n\right\}_{n=0}^{\infty}$ is positive definite,
i.e. is the moment sequence of a probability measure $\mu$ on $\mathbb{R}$,
then $\left\{(-1)^n s_n\right\}_{n=0}^{\infty}$ is the moment sequence
of the reflection $\widehat{\mu}$ of $\mu$: $\widehat{\mu}(X):=\mu(-X)$.
For the binomial sequence we have:
\begin{equation}\label{cbinomreflection}
\binom{np+r}{n}(-1)^n=\binom{n(1-p)-1-r}{n}
\end{equation}
(which in particular implies (\ref{bdreflection})),
hence if the binomial sequence (\ref{aintbinomial}), with parameters $(p,r)$, is positive definite
then it is also positive definite for parameters $(1-p,-1-r)$ and we have
\begin{equation}\label{cnureflection}
\nu(1-p,-1-r)=\widehat{\nu(p,r)}.
\end{equation}
Therefore, if either $p\ge1$, $-1\le r\le p-1$ or $p\le0$, $p-1\le r\le0$
then the binomial sequence (\ref{aintbinomial}) is positive definite
(for illustration see Figure~\ref{figura1}).
We will see in Theorem~\ref{fthcharacterisation} that the opposite implication is also true.

Let us also note relations between the measures $\nu(p,-1)$, $\nu(p,0)$, $\nu(p,p-1)$
and observe that $\nu(p,-1)$ has an atomic part.

\begin{proposition} For $p\ge1$ we have
\begin{equation}\label{cmelnuminusjeden}
\nu(p,-1)=\frac{1}{p}\delta_0+\frac{p-1}{p}\nu(p,0)
\end{equation}
and
\begin{equation}\label{cmelnupeminusjeden}
d\nu(p,p-1)(x)=\frac{x}{p}d\nu(p,0)(x).
\end{equation}
\end{proposition}

\begin{proof}
Both formulas follow directly from (\ref{bgenbinomialrelation}):
for $n\ge1$ we have
\[
\binom{np-1}{n}=\frac{p-1}{p}\binom{np}{n}=\frac{p-1}{p}\int_{0}^{c(p)} x^n\,d\mu(p,0)(x)
\]
and for $n\ge0$
\[
\binom{np+p-1}{n}=\frac{1}{p}\binom{(n+1)p}{n+1}=\frac{1}{p}\int_{0}^{c(p)} x^n x\,d\mu(p,0)(x).
\]
\end{proof}

\section{Applying Meijer $G$-function}

We know already that if $p>1$ is a rational number
and $-1<r\le p-1$ then $\nu(p,r)$ is absolutely continuous.
The aim of this section is to describe the density function
$V_{p,r}$ of $\nu(p,r)$ in terms of the Meijer $G$-function (see \cite{olver} for example)
and consequently, as a linear combination of generalized hypergeometric
functions. We will see that in some particular cases $V_{p,r}$
can be represented as an elementary function.

\begin{lemma}
For $p>1$ and $r\in\mathbb{R}$ define complex function
\begin{equation}
\psi_{p,r}(\sigma)
=\frac{\Gamma\big((\sigma-1)p+r+1\big)}{\Gamma(\sigma)\Gamma\big((\sigma-1)(p-1)+r+1\big)},
\end{equation}
where for critical $\sigma$ the right hand side is understood as the limit if exists.
Then, putting $-\mathbb{N}_0:=\{0,-1,-2,\ldots\}$, we have
\begin{equation}\label{dmeijerpsibinomial}
\psi_{p,r}(n+1)=\left\{
\begin{array}{ll}
\binom{np+r}{n}&\mbox{if }np+r+1\notin-\mathbb{N}_0,\\
{}\\
\frac{p-1}{p}\binom{np+r}{n}&\mbox{if }np+r+1\in-\mathbb{N}_0.
\end{array}
\right.
\end{equation}
\end{lemma}

\begin{proof}
If $np+r+1\notin-\mathbb{N}_0$ then the statement is a consequence
of the equality $\Gamma(z+1)=z\Gamma(z)$.
Now recall that for the reciprocal gamma function we have
\begin{equation}
\left.\frac{d}{dx}\frac{1}{\Gamma(x)}\right|_{x=-m}=(-1)^{m}m!,
\end{equation}
for $m\in\mathbb{N}_0$ (see formula (3.30) in \cite{marichev}). Therefore,
if $np+r+1=-N$, with $N\in\mathbb{N}_{0}$, then
\[
\lim_{\sigma\to n+1}\psi_{p,r}(\sigma)=\frac{(p-1)(N+n)!(-1)^{N+n}}{p n!N!(-1)^{N}}
=\frac{p-1}{p}\binom{N+n}{n}(-1)^{n}
\]
\[
=\frac{p-1}{p}\binom{-N-1}{n}=\frac{p-1}{p}\binom{np+r}{n},
\]
where we used the identity $\binom{a}{n}=\binom{n-a-1}{n}(-1)^{n}$.
\end{proof}

Note two identities which the functions $\psi_{p,r}$ satisfy:
\begin{align}
\psi_{p,-1}(\sigma)&=\frac{p-1}{p}\psi_{p,0}(\sigma),\label{dmeipsiminusjeden}\\
\psi_{p,p-1}(\sigma)&=\frac{1}{p}\psi_{p,0}(\sigma+1).\label{dmeipsipeminusjeden}
\end{align}

For rational $p>1$ and $r\in\mathbb{R}$ we define function $V_{p,r}(x)$
as the \textit{inverse Mellin transform} of $\psi_{p,r}(\sigma)$:
\begin{equation}
V_{p,r}(x)=\frac{1}{2\pi\mathrm{i}}\int_{d-\mathrm{i}\infty}^{d+\mathrm{i}\infty}
x^{-\sigma}\psi_{p,r}(\sigma)\,d\sigma
\end{equation}
whenever exists, see \cite{sneddon} for details.
Then (\ref{dmeipsiminusjeden},\ref{dmeipsipeminusjeden}) imply
\begin{align}
V_{p,-1}(x)&=\frac{p-1}{p}V_{p,0}(x),\\
V_{p,p-1}(x)&=\frac{x}{p}V_{p,0}(x).
\end{align}

It turns out that if $p>1$ is rational, $r\in\mathbb{R}$, then $V_{p,r}$
exists and can be expressed as Meijer function.

\begin{theorem}\label{dmeijertheorem}
Let $p=k/l>1$, where $k,l$ are integers such that
$1\le l<k$, and let $r\in\mathbb{R}$.
Then $V_{p,r}$ exists and can be expressed as
\begin{equation}\label{dmeijerth}
V_{p,r}(x)=\frac{p^{r+1/2}\sqrt{l}}{x(p-1)^{r+1/2}\sqrt{2\pi}}\,
G^{k,0}_{k,k}\!\left(\left.\!\!
\begin{array}{ccc}
\alpha_1,\!\!&\!\!\ldots,\!\!&\!\!\alpha_k\\
\beta_1,\!\!&\!\!\ldots,\!\!&\!\!\beta_k
\end{array}
\!\right|\frac{x^{l}}{c(p)^l}\right),
\end{equation}
$x\in(0,c(p))$, where $c(p)=p^p(p-1)^{1-p}$
and the parameters $\alpha_j,\beta_j$ are given by
(\ref{cmelalpha}) and (\ref{cmelbetha}).
Moreover, $\psi_{p,r}$ is the Mellin transform of $V_{p,r}$, i.e. we have
\begin{equation}\label{dmeizgdopsi}
\psi_{p,r}(\sigma)=\int_{0}^{c(p)} x^{\sigma-1} V_{p,r}(x)\,dx,
\end{equation}
for $\Re\sigma>1-\frac{1+r}{p}$.
\end{theorem}

\begin{proof}
Putting $m=\sigma-1$ in (\ref{cmelbinom}) we get
\begin{equation}
\psi_{p,r}(\sigma)=\frac{(p-1)^{p-r-3/2}}{p^{p-r-1/2}\sqrt{2\pi l}}
\frac{\prod_{j=1}^{k}\Gamma(\beta_j+\sigma/l-1/l)}{\prod_{j=1}^{k}\Gamma(\alpha_j+\sigma/l-1/l)}
c(p)^{\sigma}.
\end{equation}
Writing the right hand side as $\Psi(\sigma/l-1/l)c(p)^{\sigma}$,
using the substitution $\sigma=lu+1$ and the definition of the Meijer $G$-function
(see \cite{olver} for example) we obtain
\[
V_{p,r}(x)=\frac{1}{2\pi \mathrm{i}}\int_{d-\mathrm{i}\infty}^{d+\mathrm{i}\infty}
\Psi(\sigma/l-1/l)c(p)^{\sigma}x^{-\sigma}d\sigma
=\frac{lc(p)}{2\pi x\mathrm{i}}\int_{d-\mathrm{i}\infty}^{d+\mathrm{i}\infty}
\Psi(u)\left({x^l}/{c(p)^{l}}\right)^{-u}du
\]
\begin{equation}
=\frac{(p-1)^{p-r-3/2}\sqrt{l}}{z^{1/l} p^{p-r-1/2}\sqrt{2\pi}}
G^{k,0}_{k,k}\!\left(\left.\!\!
\begin{array}{ccc}
\alpha_1,\!\!&\!\!\ldots,\!\!&\!\!\alpha_k\\
\beta_1,\!\!&\!\!\ldots,\!\!&\!\!\beta_k
\end{array}
\!\right|z\right):=\widetilde{V}_{p,r}(z),
\end{equation}
$z=x^l/c(p)^l$, which leads to (\ref{dmeijerth}).
Recall that for existence of the Meijer function of type $G^{k,0}_{k,k}$
there is no restriction on the parameters $\alpha_j,\beta_j$.

On the other hand, since $\sum_{j=1}^{k}(\beta_j-\alpha_j)=-1/2<0$,
we can apply formula 2.24.2.1 from \cite{prudnikov3}.
Substituting $x:=c(p)z^{1/l}$ we have
\[
\int_{0}^{c(p)} x^{\sigma-1}V_{p,r}(x)\,dx
=\frac{c(p)^{\sigma}}{l}\int_{0}^{1} z^{\sigma/l-1}\widetilde{V}_{p,r}(z)\,dz
\]
\[
=\frac{(p-1)^{p-r-3/2}}{p^{p-r-1/2}\sqrt{2\pi l}}c(p)^{\sigma}
\int_{0}^{1}z^{\sigma/l-1/l-1} G^{k,0}_{k,k}\!\left(\left.\!\!
\begin{array}{ccc}
\alpha_1,\!\!&\!\!\ldots,\!\!&\!\!\alpha_k\\
\beta_1,\!\!&\!\!\ldots,\!\!&\!\!\beta_k
\end{array}
\!\right|z\right)\,dz=\psi_{p,r}(\sigma).
\]
In view of the assumptions for the formula 2.24.2.1 in \cite{prudnikov3},
the last equality holds provided $\Re(\sigma/l-1/l)>-\min_{j}\beta_j=-(1+r)/k$.
\end{proof}

Now we are able to describe the measures $\nu(p,r)$
for rational $p$.

\begin{corollary}
Assume that $p=k/l$, where $k,l$ are integers, $1\le l<k$.
If $-1<r\le p-1$ then the probability measure $\nu(p,r)$ is absolutely continuous
and $V_{p,r}$ is the density function, i.e.
\[
\nu(p,r)=V_{p,r}(x)\,dx,\qquad x\in(0,c(p)).
\]
For $r=-1$ we have
\[
\nu(p,-1)=\frac{1}{p}\delta_{0}+V_{p,-1}(x)\,dx,\qquad x\in(0,c(p)).
\]
\end{corollary}

\begin{proof}
This is a consequence of Theorem~\ref{cmelthmellin}, (\ref{dmeizgdopsi}), the uniqueness part
of the Riesz representation theorem for linear functionals on $\mathcal{C}[0,c(p)]$
and of the Weierstrass approximation theorem.
\end{proof}

Now applying Slater's theorem (see \cite{marichev} or (16.17.2) in \cite{olver}) we can represent $V_{p,r}$ as a linear
combination of generalized hypergeometric functions.

\begin{theorem}\label{dmeislaterth}
For $p=k/l$, with $1\le l<k$, $r\in\mathbb{R}$ and $x\in(0,c(p))$ we have
\begin{equation}\label{dmeislaterformula}
V_{p,r}(x)=\gamma(k,l,r)\sum_{h=1}^{k}c(h,k,l,r)\,
{}_{k}F_{k-1}\!\left(\left.
\begin{array}{c}
\!\!\!\mathbf{a}(h,k,l,r)\\
\!\!\!\mathbf{b}(h,k)\end{array}\!
\right|z\right)
z^{(r+h)/k-1/l},
\end{equation}
where $z=x^l/c(p)^l$,
\begin{align}
\gamma(k,l,r)&=\frac{l(p-1)^{p-r-1}}{p^{p-r-1/2}\sqrt{2\pi(k-l)}},\\
c(h,k,l,r)&=\frac{\prod_{j=1}^{h-1}\Gamma\left(\frac{j-h}{k}\right)
\prod_{j=h+1}^{k}\Gamma\left(\frac{j-h}{k}\right)}
{\prod_{j=1}^{l}\Gamma\left(\frac{j}{l}-\frac{r+h}{k}\right)
\prod_{j=l+1}^{k}\Gamma\left(\frac{r+j-l}{k-l}-\frac{r+h}{k}\right)},
\end{align}
and the parameter vectors of the hypergeometric functions are
\begin{align}
\mathbf{a}(h,k,l,r)&=
\left(\left\{\frac{r+h}{k}-\frac{j-l}{l}\right\}_{j=1}^{l},
\left\{\frac{r+h}{k}-\frac{r+j-k}{k-l}\right\}_{j=l+1}^{k}\right),\\
\mathbf{b}(h,k)&=
\left(\left\{\frac{k+h-j}{k}\right\}_{j=1}^{h-1},
\left\{\frac{k+h-j}{k}\right\}_{j=h+1}^{k}\right).
\end{align}
\end{theorem}

\begin{proof}
It is easy to check that if $i\ne j$
then the difference $\beta_i-\beta_j$
of coefficients (\ref{cmelbetha}) is not an integer and the Slater's
formula is applicable, see \cite{olver}.
\end{proof}

The easiest case is $p=2$.

\begin{corollary} For $p=2$, $r\in\mathbb{R}$ we have
\[
V_{2,r}(x)=\frac{\cos\left(r\cdot\arccos\sqrt{x/4}\right)}{\pi\sqrt{x^{1-r}(4-x)}},
\]
$x\in(0,4)$. In particular
\begin{align*}
V_{2,0}(x)&=\frac{1}{\pi\sqrt{x(4-x)}},\\
V_{2,-1/2}(x)&=\frac{1}{2\pi}\sqrt{\frac{\sqrt{x}+2}{\sqrt{x^3}(4-x)}},\\
V_{2,1/2}(x)&=\frac{1}{2\pi}\sqrt{\frac{\sqrt{x}+2}{\sqrt{x}(4-x)}},\\
V_{2,1}(x)&=\frac{\sqrt{x}}{2\pi\sqrt{4-x}}.
\end{align*}
\end{corollary}

The density $V_{2,0}$ gives the \textit{arcsine distribution} $\nu(2,0)$.
Note that if $|r|>1$ then $V_{2,r}(x)<0$ for some values of $x\in(0,4)$.

\begin{proof} We take $k=2$, $l=1$ so that $c(2)=4$, $z=x/4$ and
\[
\gamma(2,1,r)=\frac{2^r}{4\sqrt{\pi}}.
\]
Using the Euler's reflection formula, and the identity $\Gamma(1+r/2)=\Gamma(r/2)r/2$, we get
\begin{align*}
c(1,2,1,r)&=\frac{\Gamma(1/2)}{\Gamma\big((1-r)/2\big)\Gamma\big((1+r)/2\big)}
=\frac{\cos(\pi r/2)}{\sqrt{\pi}},\\
c(2,2,1,r)&=\frac{\Gamma(-1/2)}{\Gamma(-r/2)\Gamma(r/2)}=\frac{r\sin(\pi r/2)}{\sqrt{\pi}}.
\end{align*}
We also need formulas for two hypergeometric functions, namely
\begin{align*}
{}_{2}F_{1}\!\left(\left.
\frac{1+r}{2},\frac{1-r}{2};\,\frac{1}{2}\,
\right|z\right)&=\frac{\cos(r\arcsin\sqrt{z})}{\sqrt{1-z}},\\
{}_{2}F_{1}\!\left(\left.
\frac{2+r}{2},\frac{2-r}{2};\,\frac{3}{2}\,
\right|z\right)&=\frac{\sin(r\arcsin\sqrt{z})}{r\sqrt{z(1-z)}},
\end{align*}
see 15.4.12 and 15.4.16 in \cite{olver}.
Now we can write
\[
V_{2,r}(x)=
\frac{2^r\cos(r\pi/2)\cos(r\arcsin\sqrt{z})}{4\pi\sqrt{1-z}}z^{(r-1)/2}
+\frac{2^r \sin(r\pi/2)\sin(r\arcsin\sqrt{z})}{4\pi\sqrt{z(1-z)}}z^{r/2}
\]
\[
=\frac{2^r z^{(r-1)/2}}{4\pi\sqrt{1-z}}\cos\left(r\pi/2-r\arcsin\sqrt{z}\right)
=\frac{x^{(r-1)/2}\cos\left(r\cdot\arccos\sqrt{x/4}\right)}{\pi\sqrt{4-x}},
\]
which concludes the proof of the main formula.
For the particular cases we use the identity:
$\cos\left(\frac{1}{2}\arccos(t)\right)=\sqrt{(t+1)/2}$.
\end{proof}

\textbf{Remark.}
Observe that
\[
\frac{V_{2,0}\left(\sqrt{x}\right)}{2\sqrt{x}}=\frac{1}{4}V_{2,-1/2}\left(\frac{x}{4}\right).
\]
It means that if $X,Y$ are random variables such that $X\sim\nu(2,0)$ and $Y\sim\nu(2,-1/2)$
then $X^2\sim4Y$. This can be also derived from the relation $\binom{2n-1/2}{n}4^n=\binom{4n}{2n}$
(sequence A001448 in OEIS).

From (\ref{cmelnuminusjeden}) we obtain

\begin{corollary}
\[
\nu(2,-1)=\frac{1}{2}\delta_0+\frac{1}{2\pi\sqrt{x(4-x)}}\chi_{(0,4)}(x)\,dx.
\]
\end{corollary}

\section{Some special cases for $p=3$ and $p=3/2$}

From now on we are going to study some special cases for $k=3$, i.e. for $p=3$ and $p=3/2$.
Then in the formula (\ref{dmeislaterformula}) we have three  hypergeometric functions of type ${}_{3}F_{2}$
with lower parameters
\[
\mathbf{b}(1,3)=\left(\frac{1}{3},\frac{2}{3}\right),\qquad
\mathbf{b}(2,3)=\left(\frac{2}{3},\frac{4}{3}\right),\qquad
\mathbf{b}(3,3)=\left(\frac{4}{3},\frac{5}{3}\right).
\]
It turns out that
for particular choices of $r$ these hypergeometric functions reduce to
the type  ${}_{2}F_{1}$, belonging to
the following one-parameter family:
\begin{equation}
{}_{2}F_{1}\!\left(\left.
\frac{t}{2},\frac{t+1}{2};\,t\right|z\right)=\frac{\left(1+\sqrt{1-z}\right)^{1-t}}{2^{1-t}\sqrt{1-z}}
\end{equation}
(see formula~15.4.18 in \cite{olver}).

Let us start with $p=3$.

\begin{theorem}\label{etwtrzy}
For $p=3$ we have
\begin{align*}
V_{3,0}(x)
&=\frac{\left(1+\sqrt{1-z}\right)^{1/3}}{9\pi\sqrt{3(1-z)}}z^{-2/3}
+\frac{\left(1+\sqrt{1-z}\right)^{-1/3}}{9\pi\sqrt{3(1-z)}}z^{-1/3},\\
V_{3,1}(x)
&=\frac{\left({1+\sqrt{1-z}}\right)^{2/3}}{6\pi\sqrt{3(1-z)}}z^{-1/3}
+\frac{\left({1+\sqrt{1-z}}\right)^{-2/3}}{6\pi\sqrt{3(1-z)}}z^{1/3},\\
V_{3,2}(x)
&=\frac{\left(1+\sqrt{1-z}\right)^{1/3}}{4\pi\sqrt{3(1-z)}}z^{1/3}
+\frac{\left(1+\sqrt{1-z}\right)^{-1/3}}{4\pi\sqrt{3(1-z)}}z^{2/3},
\end{align*}
where $z=4x/27$ and $x\in(0,27/4)$.
\end{theorem}

\begin{proof} We have $c(3)=27/4$ and 
\begin{align*}
\gamma(3,1,r)&=\frac{2^{1-r}}{3^{5/2-r}\sqrt{\pi}},\\
c(1,3,1,r)&=\frac{2^{(r+1)/3}\sin\left(\pi(1+r)/3\right)}{\sqrt{3\pi}},\\
c(2,3,1,r)&=\frac{2^{(r-1)/3}(1-r)\sin\left(\pi(1-r)/3\right)}{\sqrt{3\pi}},\\
c(3,3,1,r)&=\frac{2^{(r-6)/3}(3-r)r\sin\left(\pi r/3\right)}{\sqrt{3\pi}}
\end{align*}
and the upper parameters for the hypergeometric functions are
\begin{align*}
\mathbf{a}(1,3,1,r)&=\left(\frac{1+r}{3},\frac{5-r}{6},\frac{2-r}{6}\right),\\
\mathbf{a}(2,3,1,r)&=\left(\frac{2+r}{3},\frac{7-r}{6},\frac{4-r}{6}\right),\\
\mathbf{a}(3,3,1,r)&=\left(\frac{3+r}{3},\frac{9-r}{6},\frac{6-r}{6}\right).
\end{align*}

For $r=0$ we have $c(3,3,1,0)=0$ and
\[
V_{3,0}(x)=\frac{2}{9\sqrt{3\pi}}\left(\frac{2^{1/3}}{2\sqrt{\pi}}
\frac{\left(1+\sqrt{1-z}\right)^{1/3}}{2^{1/3}\sqrt{1-z}}z^{-2/3}
+\frac{2^{-1/3}}{2\sqrt{\pi}}\frac{\left(1+\sqrt{1-z}\right)^{-1/3}}{2^{-1/3}\sqrt{1-z}}z^{-1/3}
\right)
\]
\[
=\frac{\left(1+\sqrt{1-z}\right)^{1/3}}{9\pi\sqrt{3(1-z)}}z^{-2/3}
+\frac{\left(1+\sqrt{1-z}\right)^{-1/3}}{9\pi\sqrt{3(1-z)}}z^{-1/3}.
\]

For $r=1$ the second term vanishes and 
\[
V_{3,1}(x)=\frac{1}{3\sqrt{3\pi}}
\left(\frac{2^{2/3}}{2\sqrt{\pi}}\frac{\left({1+\sqrt{1-z}}\right)^{2/3}}{2^{2/3}\sqrt{1-z}}z^{-1/3}
+\frac{2^{-5/3}}{\sqrt{\pi}}\frac{\left({1+\sqrt{1-z}}\right)^{-2/3}}{2^{-2/3}\sqrt{1-z}}z^{1/3}\right)
\]
\[
=\frac{\left({1+\sqrt{1-z}}\right)^{2/3}}{6\pi\sqrt{3(1-z)}}z^{-1/3}
+\frac{\left({1+\sqrt{1-z}}\right)^{-2/3}}{6\pi\sqrt{3(1-z)}}z^{1/3}.
\]
Finally, for the third formula we can apply (\ref{cmelnupeminusjeden}),
which gives us $V_{3,2}(x)=\frac{9z}{4}V_{3,0}(x)$.
\end{proof}

Similarly we work with $p=3/2$.

\begin{theorem}\label{etwtrzydrugie}
For $p=3/2$ we have
\begin{align*}
V_{3/2,-1/2}(x)
&=\frac{\left({1+\sqrt{1-z}}\right)^{2/3}}{3\pi\sqrt{3(1-z)}}z^{-1/3}
+\frac{\left({1+\sqrt{1-z}}\right)^{-2/3}}{3\pi\sqrt{3(1-z)}}z^{1/3},\\
V_{3/2,0}(x)
&=
\frac{\left(1+\sqrt{1-z}\right)^{1/3}}{3\pi\sqrt{1-z}}z^{-1/6}
+\frac{\left({1+\sqrt{1-z}}\right)^{-1/3}}{3\pi\sqrt{1-z}}z^{1/6},\\
V_{3/2,1/2}(x)
&=
\frac{\left({1+\sqrt{1-z}}\right)^{1/3}}{\pi\sqrt{3(1-z)}}z^{1/3}
+\frac{\left({1+\sqrt{1-z}}\right)^{-1/3}}{\pi\sqrt{3(1-z)}}z^{2/3},
\end{align*}
where $z=4x^2/27$, $x\in(0,\sqrt{27/4})$.
\end{theorem}

\begin{proof}
We have $c(3/2)=\sqrt{27}/2$,
\begin{align*}
\gamma(3,2,r)&=\frac{2\cdot 3^r}{3\sqrt{\pi}},\\
c(1,3,2,r)&=\frac{2^{(1-2r)/3}\sin\left(\pi(1-2r)/3\right)}{\sqrt{3\pi}},\\
c(2,3,2,r)&=\frac{2^{(-1-2r)/3}(1+2r)\sin\left(\pi(1+2r)/3\right)}{\sqrt{3\pi}},\\
c(3,3,2,r)&=\frac{2^{(-3-2r)/3}(3+2r)r\sin\left(2\pi r/3\right)}{\sqrt{3\pi}}
\end{align*}
and
\begin{align*}
\mathbf{a}(1,3,2,r)&=\left(\frac{5+2r}{6},\frac{1+r}{3},\frac{1-2r}{3}\right),\\
\mathbf{a}(2,3,2,r)&=\left(\frac{7+2r}{6},\frac{2+r}{3},\frac{2-2r}{3}\right),\\
\mathbf{a}(3,3,2,r)&=\left(\frac{9+2r}{6},\frac{3+r}{3},\frac{3-2r}{3}\right).
\end{align*}

For $r=-1/2$ the second term vanishes and we have
\[
V_{3/2,-1/2}(x)=\frac{2}{3\sqrt{3\pi}}
\left(\frac{2^{2/3}}{2\sqrt{\pi}}\frac{\left({1+\sqrt{1-z}}\right)^{2/3}}{2^{2/3}\sqrt{1-z}}z^{-1/3}
+\frac{2^{-2/3}}{2\sqrt{\pi}}\frac{\left({1+\sqrt{1-z}}\right)^{-2/3}}{2^{-2/3}\sqrt{1-z}}z^{1/3}\right)
\]
\[
=\frac{\left({1+\sqrt{1-z}}\right)^{2/3}}{3\pi\sqrt{3(1-z)}}z^{-1/3}
+\frac{\left({1+\sqrt{1-z}}\right)^{-2/3}}{3\pi\sqrt{3(1-z)}}z^{1/3}.
\]

For $r=0$ we note that $c(3,3,2,0)=0$ and we get
\[
V_{3/2,0}(x)=\frac{2}{3\sqrt{\pi}}\left(\frac{2^{1/3}}{2\sqrt{\pi}}
\frac{\left(1+\sqrt{1-z}\right)^{1/3}}{2^{1/3}\sqrt{1-z}}z^{-1/6}
+\frac{2^{-1/3}}{2\sqrt{\pi}}\frac{\left(1+\sqrt{1-z}\right)^{-1/3}}{2^{-1/3}\sqrt{1-z}}z^{1/6}
\right)
\]
\[
=
\frac{\left(1+\sqrt{1-z}\right)^{1/3}}{3\pi\sqrt{1-z}}z^{-1/6}
+\frac{\left({1+\sqrt{1-z}}\right)^{-1/3}}{3\pi\sqrt{1-z}}z^{1/6}.
\]
Finally, by (\ref{cmelnupeminusjeden}) we have
$V_{3/2,1/2}(x)=\sqrt{3z}V_{3/2,0}(x)$, which leads to the third formula.
\end{proof}

Let us mention that the integer sequence $\binom{3n/2-1/2}{n}4^n$:
\[
1,4,30,256,2310,21504,204204,1966080,19122246,\ldots
\]
appears in OEIS as A091527.
It is the moment sequence of the density function
\begin{equation}
V(x):=V_{3/2,-1/2}(x/4)/4
\end{equation}
\[
=\frac{x^{4/3}+9\cdot 2^{4/3}\left(1+\sqrt{1-x^2/108}\right)^{4/3}}
{2^{8/3}\cdot 3^{5/2}\cdot\pi\cdot x^{2/3}\sqrt{1-x^2/108} \left(1+\sqrt{1-x^2/108}\right)^{2/3}}
\]
on $(0,6\sqrt{3})$ (see Theorem~\ref{etwtrzydrugie})
and its generating function is $\mathcal{D}_{3/2,-1/2}(4z)$ (see Corollary~\ref{bcortrzydrugie}).
The even numbered terms constitute sequence A061162, i.e.
\[
\mathrm{A091527}(2n)=\mathrm{A061162}(n)=\binom{3n-1/2}{2n}16^n=\frac{(6n)!n!}{(3n)!(2n)!^2},
\]
so this is the moment sequence for the density function
\begin{equation}
\frac{V\left(\sqrt{x}\right)}{2\sqrt{x}}
=\frac{x^{2/3}+9\cdot 2^{4/3}\left(1+\sqrt{1-x/108}\right)^{4/3}}
{2^{11/3}\cdot 3^{5/2}\cdot\pi\cdot x^{5/6}\sqrt{1-x/108} \left(1+\sqrt{1-x/108}\right)^{2/3}}
\end{equation}
on the interval $(0,108)$.

From (\ref{cmelnuminusjeden}) we can also write down the measures $\nu(3,-1)$ and $\nu(3/2,-1)$:
\[
\nu(3,-1)=\frac{1}{3}\delta_0+\frac{2}{3}V_{3,0}(x)\,dx,
\]
\[
\nu(3/2,-1)=\frac{2}{3}\delta_0+\frac{1}{3}V_{3/2,0}(x)\,dx.
\]

\section{Some convolution relations}

In this part we are going to prove a few formulas involving
the measures $\nu(p,r)$, the measures $\mu(p,r)$ studied in
\cite{mlotkowski2010,mpz2012}, and various types of convolutions.
First we observe that the families
$\nu(p,r)$ and $\mu(p,r)$ are related through the Mellin convolution.

\begin{proposition}\label{fpropmellin}
For $c>0$ define probability measure $\eta(c)$ by
\[
\eta(c):=c\cdot x^{c-1}\,dx,\quad x\in[0,1].
\]
Then for $p>1$, $0<r\le p$ we have
\[
\nu(p,r-1)\circ\eta\left(r/(p-1)\right)=\mu(p,r).
\]
\end{proposition}

\begin{proof}
Since the moment sequence of $\eta(c)$ is $\left\{\frac{c}{n+c}\right\}_{n=0}^{\infty}$,
it is sufficient to note that
\[
\binom{np+r-1}{n}\cdot\frac{r}{n(p-1)+r}=\binom{np+r}{n}\frac{r}{np+r}.
\]
\end{proof}

Note that Theorem~5.1 in \cite{mlotkowski2010} is a consequence
of Theorem~\ref{cthepositive} and Proposition~\ref{fpropmellin}.

For a compactly supported probability measure $\mu$ on $\mathbb{R}$
define its \textit{moment generating function} by
\[
M_{\mu}(z):=\int_{\mathbb{R}}\frac{1}{1-xz}\,d\mu(x).
\]
For two such measures $\mu_1,\mu_2$ we define their
\textit{monotonic convolution} (due to Muraki \cite{muraki})
$\mu_1\vartriangleright\mu_2$ by
\[
M_{\mu_1\vartriangleright\mu_2}(z)=M_{\mu_1}\big(zM_{\mu_2}(z)\big)\cdot M_{\mu_2}(z).
\]
It is an associative, noncommutative operation on probability measures on $\mathbb{R}$.

Let $\mu$ be a probability measure with compact support contained
in the positive halfline $[0,+\infty)$ and $\mu\ne\delta_{0}$.
Then the \textit{$S$-transform}, defined by
\[
M_{\mu}\left(\frac{z}{1+z}S_{\mu}(z)\right)=1+z
\]
can be used to describe the dilation $\mathbf{D}_{c}\mu$,
the \textit{multiplicative free power} $\mu^{\boxtimes p}$,
the \textit{additive free power} $\mu^{\boxplus t}$
and the \textit{Boolean power} $\mu^{\uplus u}$ by:
\begin{align}
S_{\mathbf{D}_c\mu}(z)&=\frac{1}{c}S_{\mu}(z),\\
S_{\mu^{\boxtimes p}}(z)&=S_{\mu}(z)^{p},\\
S_{\mu^{\boxplus t}}(z)&=\frac{1}{t}S_{\mu}\left(\frac{z}{t}\right),\\
S_{\mu^{\uplus u}}(z)&=\frac{1}{u}S_{\mu}\left(\frac{z}{u+(u-1)z}\right).
\end{align}
These measures are well defined
(and have compact support
contained in the positive half-line) for $c,u>0$ and at least for $p,t\ge1$.
The Boolean power can be also described through the moment generating function:
\[
M_{\mu^{\uplus u}}(z)=\frac{M_{\mu}(z)}{u-(u-1)M_{\mu}(z)}.
\]
For example, $S_{\mu(p,1)}(z)=(1+z)^{1-p}$, which implies
that $\mu(p,1)=\mu(2,1)^{\boxtimes p-1}$.

Recall from \cite{mlotkowski2010} three formulas, which are
consequences of (\ref{bgenbdgenfunct}) and of the identity:
\[
\mathcal{B}_{p-r}\left(z\mathcal{B}_{p}(z)^r\right)=\mathcal{B}_{p}(z).
\]

\begin{proposition}\label{gpropbernoulli}
For $p\ge1$, $-1\le r\le p-1$ and $0\le a,b,s\le p$ we have
\begin{align}
\nu(p,0)=\mu(p,1)^{\uplus p}&=\left(\mu(2,1)^{\boxtimes p-1}\right)^{\uplus p},\\
\mu(p,a)\vartriangleright\mu(p+b,b)&=\mu(p+b,a+b),\label{fmumonotonic}\\
\nu(p,r)\vartriangleright\mu(p+s,s)&=\nu(p+s,r+s).\label{fnumonotonic}
\end{align}
In particular, if $0\le r\le p-1$ then
\[
\nu(p,r)=\nu(p-r,0)\vartriangleright\mu(p,r)
=\mu(p-r,1)^{\uplus p-r}\vartriangleright\mu(p,r).
\]
\end{proposition}

The formulas (\ref{fmumonotonic}) and (\ref{fnumonotonic})
will be applied in the next section.

Finally we observe that $\nu(p,-1)$ and $\nu(p,0)$
are free convolutions powers of the Bernoulli distribution.

\begin{proposition}
For $p>1$ we have
\begin{align}
\nu(p,-1)&=\mathbf{D}_{c(p)}
\left(\frac{1}{p}\delta_{0}+\frac{p-1}{p}\delta_{1}\right)^{\boxtimes p},\\
\intertext{where $c(p)=p^p(p-1)^{1-p}$, and}
\nu(p,0)&=\mathbf{D}_{p}
\left(\left(\frac{1}{p}\delta_{0}+\frac{p-1}{p}\delta_{1}\right)^{\boxplus p/(p-1)}\right)^{\boxtimes p-1}.
\end{align}
\end{proposition}

\begin{proof}
First we prove that
\begin{align}
S_{\nu(p,-1)}(z)&=\frac{(p-1)^{p-1}}{p^p}\left(\frac{1+z}{\frac{p-1}{p}+z}\right)^p,\\
S_{\nu(p,0)}(z)&=\frac{(p-1)^{p-1}}{p^p}\left(\frac{\frac{p}{p-1}+z}{1+z}\right)^{p-1}.
\end{align}
Indeed, putting $w=\frac{z}{(p-1)(1+z)}$ we have $1+w=\frac{pz+p-1}{(p-1)(1+z)}$.
Therefore
\[
\mathcal{D}_{p,-1}\left(w(1+w)^{-p}\right)=\frac{1}{p-(p-1)(1+w)}=1+z.
\]
Similarly, putting $v=\frac{z}{p-z+pz}$ we get $1+v=\frac{p(1+z)}{p-z+pz}$ and
\[
\mathcal{D}_{p,0}\left(v(1+v)^{-p}\right)=\frac{1+v}{p-(p-1)(1+v)}=1+z.
\]

Now it remains to note that the $S$-transform of
$\alpha \delta_{0}+(1-\alpha)\delta_{a}$, with $0<\alpha<1$, $a>0$,
is $\frac{1+z}{a(1-\alpha+z)}$ and recall that $S_{\mu(2,1)}(z)=(1+z)^{-1}$.
\end{proof}

Let us note two particular cases:
\begin{equation}
\nu(2,-1)=\mathbf{D}_{4}\left(\frac{1}{2}\delta_0+\frac{1}{2}\delta_{1}\right)^{\boxtimes 2}
\qquad\hbox{and}\qquad
\nu(2,0)=\mathbf{D}_{2}\left(\frac{1}{2}\delta_0+\frac{1}{2}\delta_{1}\right)^{\boxplus 2}.
\end{equation}

\section{The necessary conditions for positive definiteness}

This section is fully devoted to the necessary conditions for the positive definiteness
of the binomial (\ref{aintbinomial}) and Raney (\ref{aintraney}) sequences.

For the Raney sequence (\ref{aintraney}) we have
\begin{equation}
\binom{np+r}{n}\frac{r(-1)^n}{np+r}=\binom{n(1-p)-r}{n}\frac{-r}{n(1-p)-r}
\end{equation}
which yields (\ref{bbreflection}) and
\begin{equation}
\mu(1-p,-r)=\widehat{\mu(p,r)}.
\end{equation}
Therefore, if either
$p\ge1$, $0\le r\le p$ or $p\le0$, $p-1\le r\le0$
then the Raney sequence (\ref{aintraney}) is positive definite.
In addition, if $r=0$ then the sequence
is just $(1,0,0,\ldots)$, the moment sequence of $\delta_0$.

We are going to prove that these statements fully
characterize positive definiteness of these sequences.
We will need the following

\begin{lemma}\label{flemma}
Let $\left\{s_n\right\}_{n=0}^{\infty}$ be a sequence
of real numbers and let $w:(a,b)\to\mathbb{R}$ be a continuous function,
such that $w(x_0)<0$ for some $x_0\in(a,b)$
and there is $N$ such that $s_n=\int_{a}^{b} x^n w(x)\,dx$
for $n\ge N$. Then $\left\{s_n\right\}_{n=0}^{\infty}$
is not positive definite.
\end{lemma}

\begin{proof}
Put $w_0(x):=w(x)x^{2N}$ and $a_n:=s_{2N+n}$.
Then $a_n:=\int_{a}^{b}w_{0}(x) x^{n}\,dx$ and
in view of the the uniqueness part of the Riesz representation theorem
for linear functionals on $\mathcal{C}[a,b]$
and of the Weierstrass approximation theorem,
the sequence $\left\{a_n\right\}_{n=0}^{\infty}$
is not positive definite. Since
\[
\sum_{i,j}a_{i+j}\alpha_{i}\alpha_{j}
=\sum_{i,j}s_{(N+i)+(N+j)}\alpha_{i}\alpha_{j},
\]
the sequence $\left\{s_n\right\}_{n=0}^{\infty}$
is not positive definite too.
\end{proof}

Now we are ready to prove the main result of this section.
Note that formulas (\ref{fmumonotonic}) and (\ref{fnumonotonic})
play key role in the proof.

\begin{theorem}\label{fthcharacterisation}
1. The binomial sequence (\ref{aintbinomial}) is positive definite if and only if either
$p\ge1$, $-1\le r\le p-1$ or $p\le0$, $p-1\le r\le0$.

2. The Raney sequence (\ref{aintraney}) is positive definite if and only if either
$p\ge1$, $0\le r\le p$ or $p\le0$, $p-1\le r\le0$ or else $r=0$.
\end{theorem}

\begin{figure}
\caption{The set of pairs $(p,r)$ for which the binomial sequence $\binom{np+r}{n}$
is positive definite. The dots and the vertical line at $p=2$ indicate
those parameters $(p,r)$ for which the density $V_{p,r}$ is an elementary function.}
\centering
\includegraphics[width=0.65\textwidth]{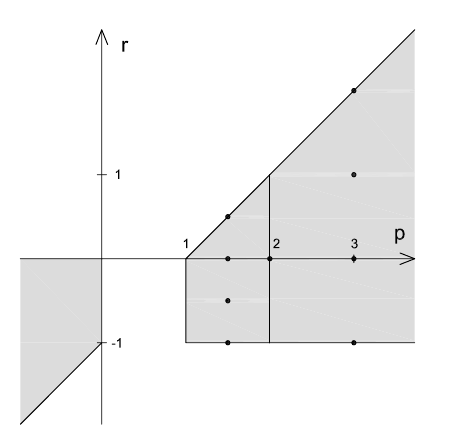}
\label{figura1}
\end{figure}

\begin{proof}
We need only to prove that these conditions are necessary.

First we note that the maps $(p,r)\mapsto(1-p,-1-r)$
and $(p,r)\mapsto(1-p,-r)$ are reflections with respect to
the points $(1/2,-1/2)$ and $(1/2,0)$ respectively.
Therefore it is sufficient to confine ourselves to $p\ge1/2$.
First we will show that if the binomial sequence (\ref{aintbinomial})
(resp. the Raney sequence \ref{aintraney}) is positive definite
and $p\ge1/2$ then $p\ge1$ (or $r=0$ in the case or Raney sequence).

We will use the fact that if a sequence $\left\{s_{n}\right\}_{n=0}^{\infty}$ is positive definite then
$s_{2n}\ge0$ and $\det \left(s_{i+j}\right)_{i,j=0}^{n}\ge0$ for all $n\in\mathbb{N}$.
Then the condition $\det \left(s_{i+j}\right)_{i,j=0}^{1}\ge0$
is equivalent to
\begin{align}
2p^2-2p-r-r^2&\ge0,\label{detbinom}\\
\intertext{for the binomial sequence (\ref{aintbinomial}) and}
r(2p-r-1)&\ge0\label{detraney}
\end{align}
for the Raney sequence (\ref{aintraney}).

Assume that $1/2\le p<1$, $r\in\mathbb{R}$
and that the binomial sequence (\ref{aintbinomial}) is positive definite.
Then (\ref{detbinom}) implies that $-1\le r\le0$
and $\frac{1}{2}+\frac{1}{2}\sqrt{1+r+r^2}\le p<1$,
which, in turn, implies $\frac{1}{2}+\frac{\sqrt{3}}{4}\le p<1$.
Hence $0<1-p<1/2$ and there exists $n_0$ such that
\[
r+1<n_0(1-p)<(n_0+1)(1-p)<r+2.
\]
This implies
\[
\binom{n_0 p+r}{n_0}=\frac{1}{n_0 !}\prod_{i=1}^{n_0}\big(n_0(p-1)+r+i\big)<0
\]
and
\[
\binom{(n_0+1)p+r}{n_0+1}=\frac{1}{(n_0+1)!}\prod_{i=1}^{n_0+1}\big((n_0+1)(p-1)+r+i\big)<0
\]
(the first factor is negative, all the others are positive),
which contradicts positive definiteness of the sequence
because one of the numbers $n_0,n_0+1$ is even.

Similarly, if  $1/2\le p<1$, $r\ne0$
and the Raney sequence is positive definite then
(\ref{detraney}) implies that $1/2<p<1$ and $0<r\le 2p-1$.
Hence we can choose $n_0$ in the same way as before, so that
\[
\binom{n_0 p+r}{n_0}\frac{r}{n_0 p+r}=\frac{r}{n_0 !}\prod_{i=1}^{n_0-1}\big(n_0(p-1)+r+i\big)<0
\]
and
\[
\binom{(n_0+1)p+r}{n_0+1}\frac{r}{(n_0+1)p+r}=\frac{r}{(n_0+1) !}\prod_{i=1}^{n_0}\big((n_0+1)(p-1)+r+i\big)<0,
\]
which contradicts positive definiteness of the Raney sequence (\ref{aintraney}).

So far we have proved that if $p\ge1/2$ and the binomial sequence (\ref{aintbinomial})
(resp. the Raney sequence (\ref{aintraney}))
is positive definite (and if $r\ne0$ for the case of the Raney sequence) then $p\ge1$.
For $p=1$ the conditions (\ref{detbinom}) and (\ref{detraney}) imply that
$-1\le r\le0$ and $0\le r\le1$ respectively.
From now on we will assume that $p>1$.

Now we will work with the Raney sequences.
Denote by $\Sigma_{R}$ the set of all pairs $(p,r)$ such that
$p\ge1$ and (\ref{aintraney}) is positive definite. By (\ref{detraney}),
if $(p,r)\in\Sigma_{R}$ then $r\ge0$.

Recall that if $p=k/l$, $1\le l<k$ and $r\ne0$ then we have
\[
\binom{np+r}{n}\frac{r}{np+r}=\int_{0}^{c(p)} x^n\cdot W_{p,r}(x)\, dx,
\]
$n=0,1,2,\ldots$, where for $W_{p,r}$ we have the following expression:
\begin{equation}\label{meislater}
W_{p,r}(x)=\widetilde{\gamma}(k,l,r)\sum_{h=1}^{k}\widetilde{c}(h,k,l,r)\,
{}_{k}F_{k-1}\!\left(\left.
\begin{array}{c}
\!\!\!\widetilde{\mathbf{a}}(h,k,l,r)\\
\!\!\!\widetilde{\mathbf{b}}(h,k,l,r)\end{array}\!
\right|z\right)
z^{(r+h-1)/k-1/l},
\end{equation}
where $z=x^l/c(p)^l$,
\begin{align}
\widetilde{\gamma}(k,l,r)&=\frac{r(p-1)^{p-r-3/2}}{p^{p-r}\sqrt{2\pi k }},\\
\widetilde{c}(h,k,l,r)&=\frac{\prod_{j=1}^{h-1}\Gamma\left(\frac{j-h}{k}\right)
\prod_{j=h+1}^{k}\Gamma\left(\frac{j-h}{k}\right)}
{\prod_{j=1}^{l}\Gamma\left(\frac{j}{l}-\frac{r+h-1}{k}\right)
\prod_{j=l+1}^{k}\Gamma\left(\frac{r+j-l}{k-l}-\frac{r+h-1}{k}\right)}.
\end{align}
This is a consequence of the proof of Theorem in \cite{mpz2012}.

Now we are going to prove that if $p=k/l>1$ and $p<r<2p$
then $W_{p,r}(x)<0$ for some $x>0$.
Indeed we have $-1<\frac{1}{l}-\frac{r}{k}<0$, $\frac{j}{l}-\frac{r}{k}>0$ for $j\ge2$
and $\frac{r+j-l}{k-l}-\frac{r+h-1}{k}>0$ for $j>l$.
Therefore  $\widetilde{c}(1,k,l,r)<0$ and then we can express $W_{p,r}$ as
\[
W_{p,r}(x)=x^{r/p-1}\left[\sum_{h=1}^{k}\phi_{h}(x)x^{(h-1)/p}\right],
\]
where $\phi_{h}(x)$ are continuous functions on $[0,c(p))$ and $\phi_{1}(0)<0$.
This implies that $W_{p,r}(x)<0$ if $x>0$ is sufficiently small
and therefore the sequence is not positive definite in this case

On the other hand, if $(p_0,r_0)\in\Sigma_{R}$, $t>0$ then
in view of (\ref{fmumonotonic}) we have $(p_{0}+t,r_{0}+t)\in\Sigma_{R}$,
namely
\[
\mu(p_0+t,r_0+t)=\mu(p_0,r_0)\rhd\mu(p_0+t,t_0).
\]
Hence, if we had $p_0<r_0$ then we could chose $t>0$
such that $p_0+t<r_0+t<2(p_0+t)$ and $p_0+t$ is rational,
which in turn implies that $(p_0+t,r_0+t)\notin\Sigma_{R}$.
This contradiction concludes the proof that $\Sigma_{R}=\{(p,r):p\ge1\hbox{ and } 0\le r\le p\}$.

Denote by $\Sigma_{b}$ the set of all pairs $(p,r)$ such that
$p\ge1$ and the binomial sequence (\ref{aintbinomial}) is positive definite.

If $p=k/l>1$ and $-p-1<r<-1$ then $\frac{j}{l}-\frac{r+1}{k}>0$ for $1\le j\le l$,
$-1<\frac{r+1}{k-l}-\frac{r+1}{k}<0$ and $\frac{r+j-l}{k-l}-\frac{r+1}{k}>0$ for $l+2\le j\le k$.
Therefore in formula (\ref{dmeislaterformula}) we have ${c}(1,k,l,r)<0$ and
$V_{p,r}(x)<0$ if $x>0$ is sufficiently small.
Consequently, in view of (\ref{dmeijerpsibinomial}), (\ref{dmeizgdopsi})
and Lemma~\ref{flemma}, $(p,r)\notin\Sigma_{b}$.

Now we can prove that if $(p_0,r_0)\in\Sigma_{b}$ then $r_0\ge-1$.
Indeed, if $r_0<-1$ then we can find
$t\ge 0$ such that $-p_0-t-1<r_0+t<-1$ and $p_0+t$ is rational.
By (\ref{fnumonotonic}) we have $(p_0+t,r_0+t)\in\Sigma_{b}$,
because
\[
\nu(p_0+t,r_0+t)=\nu(p_0,r_0)\rhd\mu(p_0+t,t_0),
\]
which is in contradiction with the previous paragraph.

Similarly we prove that if $(p_0,r_0)\in \Sigma_{b}$ then $r_0\le p_0-1$.
If $p=k/l>1$ and $p-1<r<2p-1$ then
$-1<\frac{1}{l}-\frac{r+1}{k}<0$, $\frac{j}{l}-\frac{r+1}{k}>0$ for $j\ge2$
and $\frac{r+j-l}{k-l}-\frac{r+1}{k}>0$ for $l+1\le j\le k$.
Hence $c(1,k,l,r)<0$, $V_{p,r}(x)<0$ if $x>0$ is small enough
and $(p,r)\notin\Sigma_{b}$.

Now, if $(p_0,r_0)\in\Sigma_{b}$, $p_0\ge0$ and $r_0>p_{0}-1$
then one can find some $t>0$ such that $r_{0}+t<2p_0+2t-1$ and $p_0+t$ is a rational number.
Then, in view of (\ref{fnumonotonic}), $(p_0+t,r_0+t)\in\Sigma_{b}$, which is in contradiction with the previous paragraph.
This concludes the whole proof.
\end{proof}

\begin{figure}
\caption{Density functions $V_{p,0}(x)$ for some values of $p$.}
\centering
\includegraphics[width=0.65\textwidth]{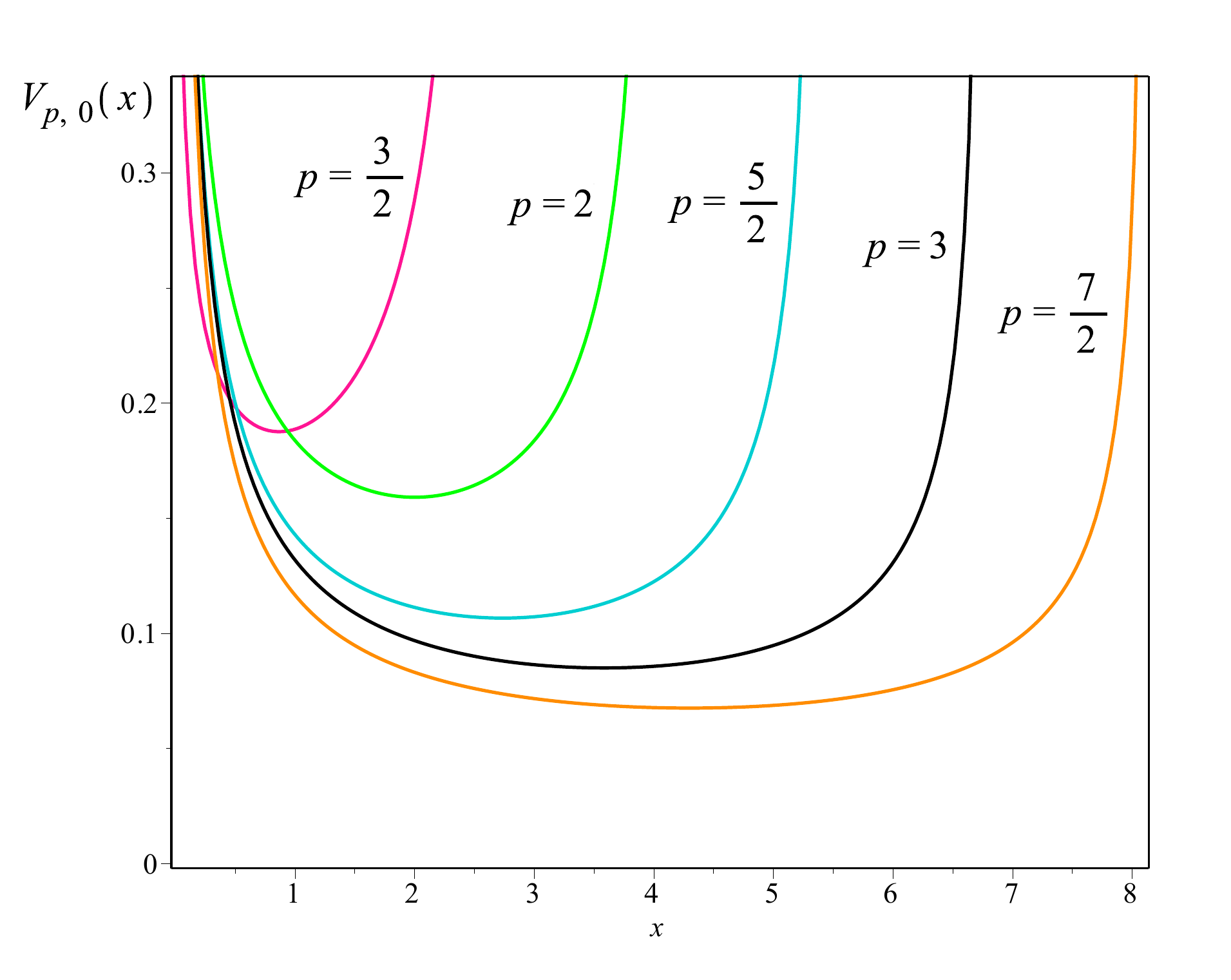}
\label{figura2}
\end{figure}

\begin{figure}
\caption{Some examples of $V_{p,r}$ for $p=3/2$.
Note that for $r=3/4$ and $1$ the function $V_{3/2,r}(x)$ has also negative values.}
\centering
\includegraphics[width=0.65\textwidth]{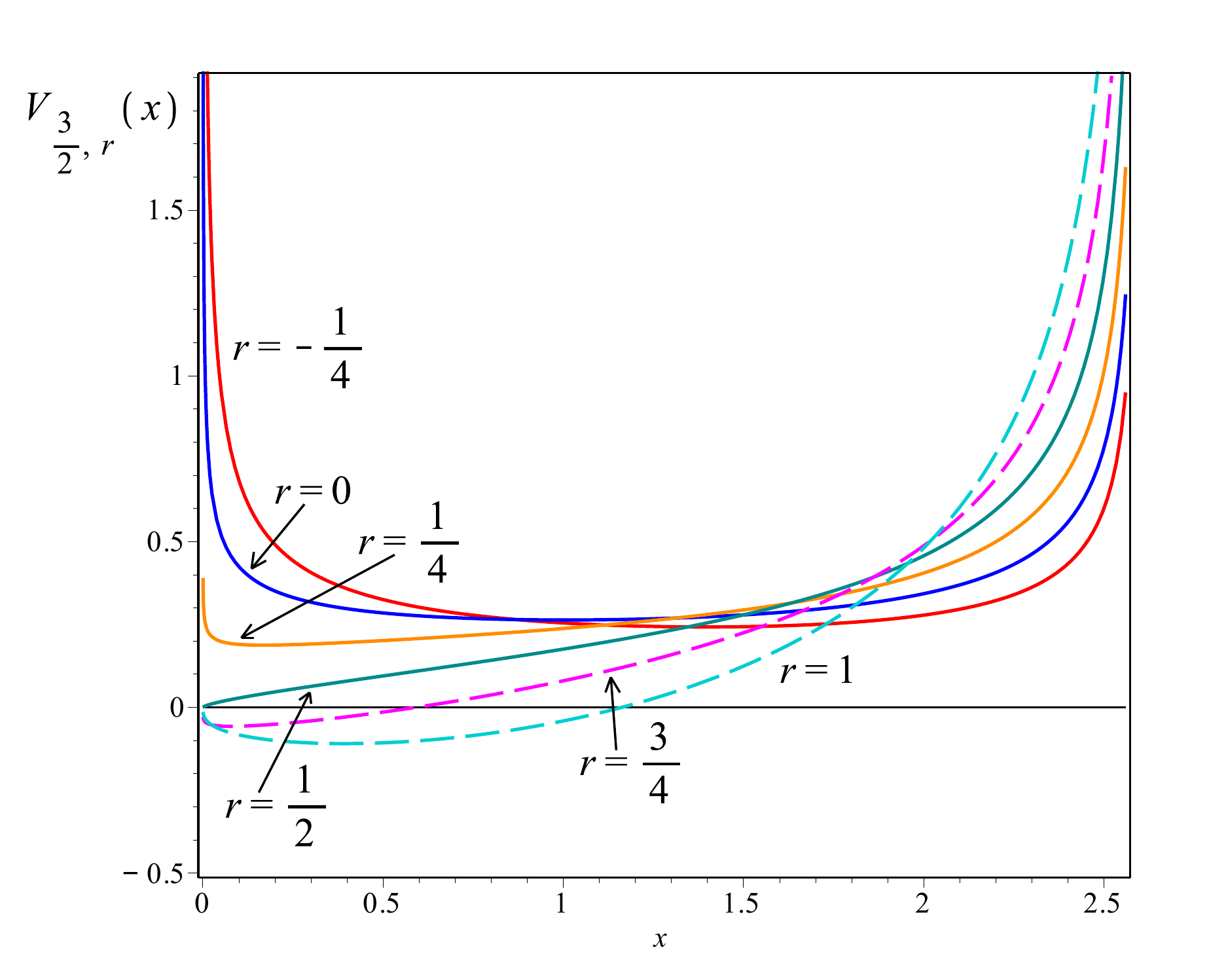}
\label{figura3}
\end{figure}

\begin{figure}
\caption{Some examples of $V_{p,r}(x)$ for $p=5/3$.}
\centering
\includegraphics[width=0.65\textwidth]{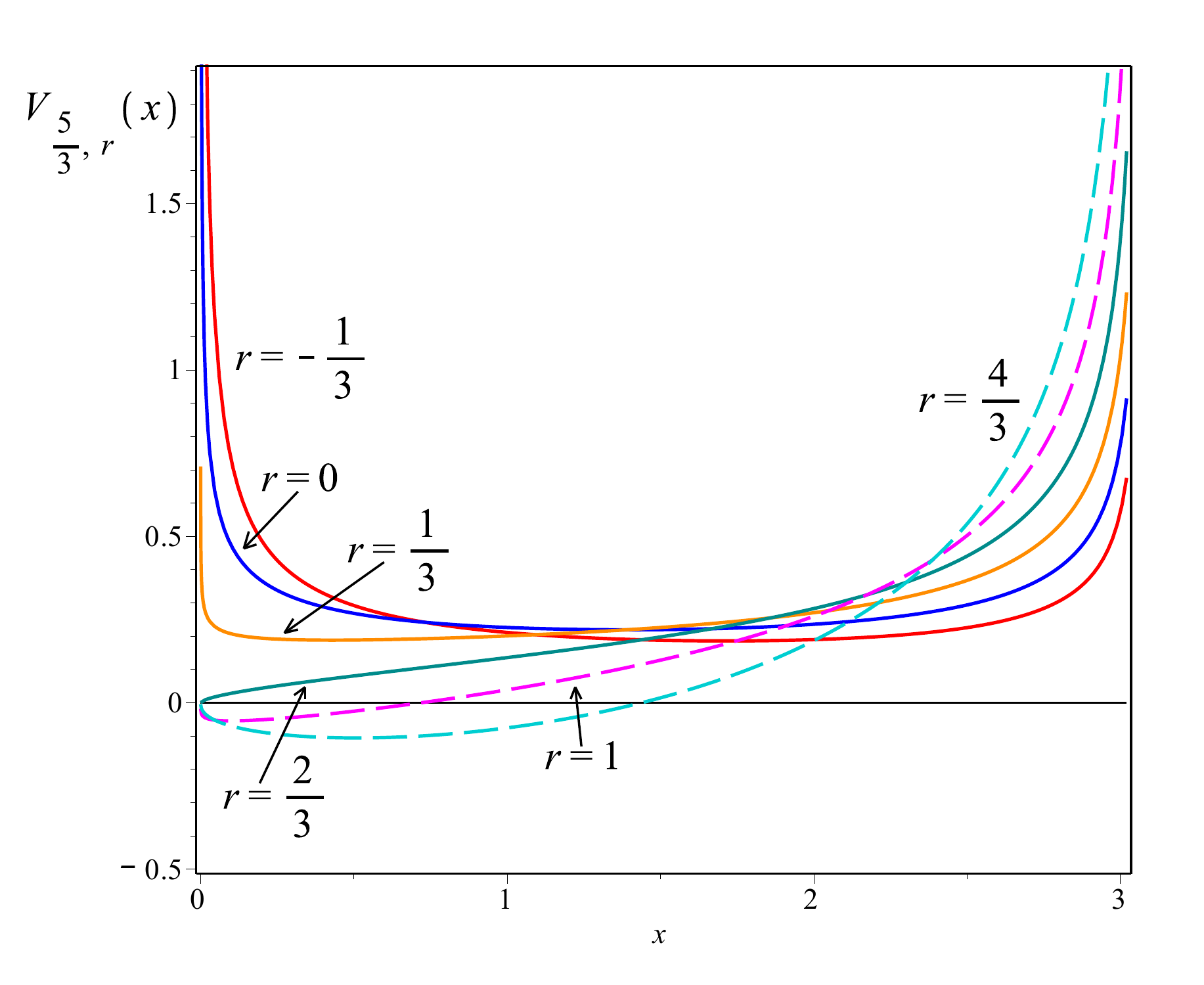}
\label{figura4}
\end{figure}

\begin{figure}
\caption{Some examples of $V_{p,r}(x)$ for $p=7/2$.}
\centering
\includegraphics[width=0.65\textwidth]{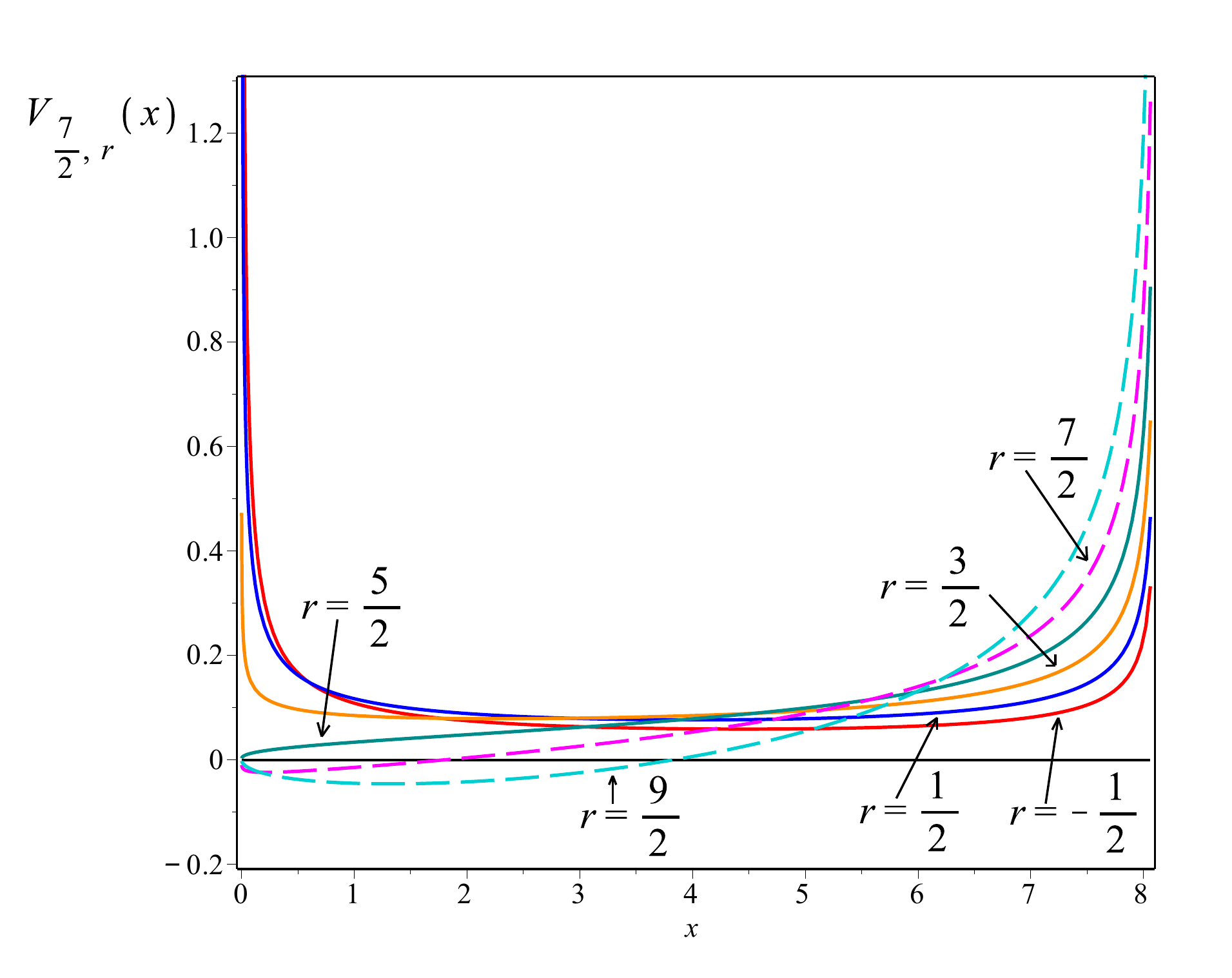}
\label{figura5}
\end{figure}

\begin{figure}
\caption{Some examples of $V_{p,r}(x)$ for $r=-3/2$.}
\centering
\includegraphics[width=0.65\textwidth]{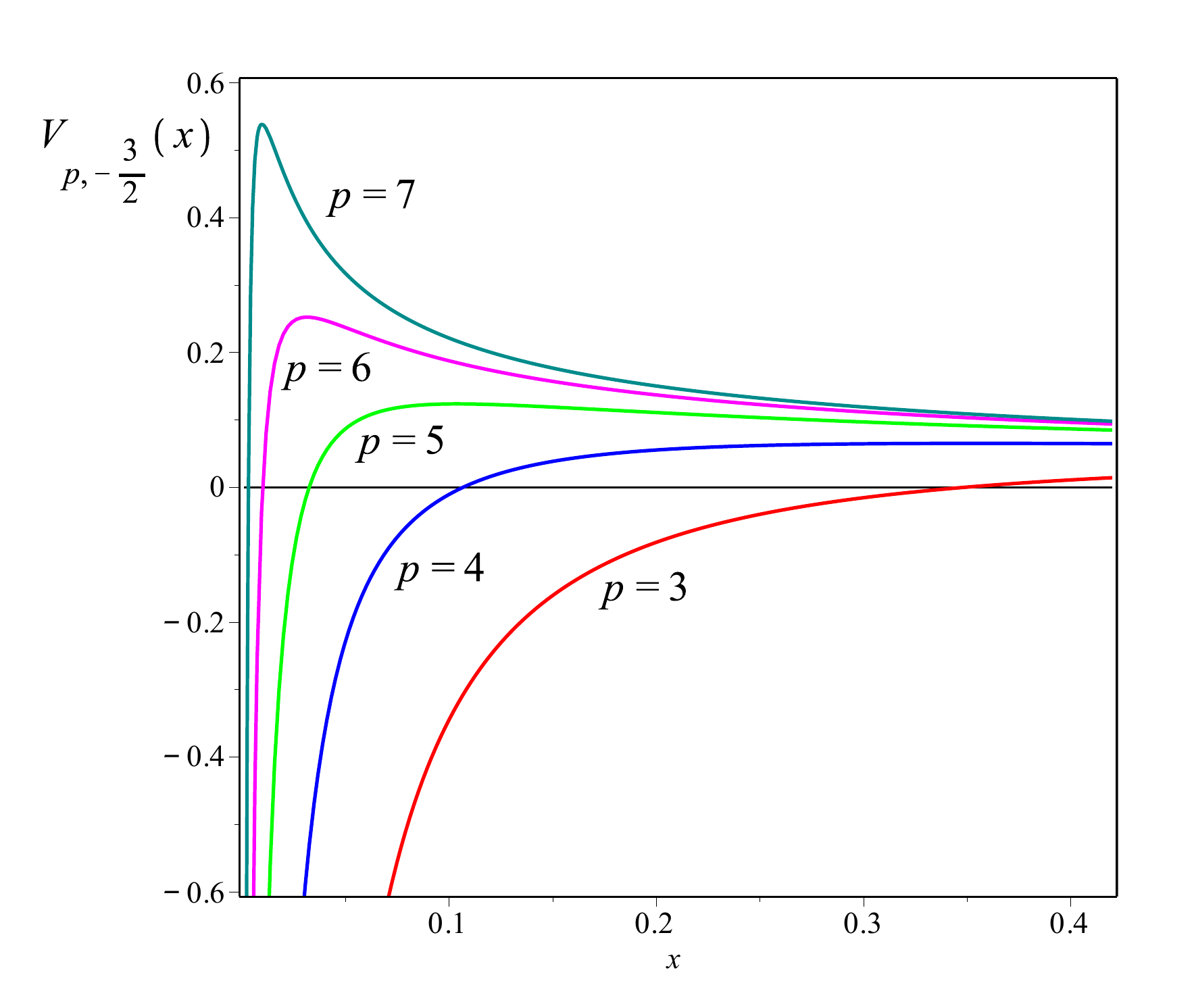}
\label{figura6}
\end{figure}

\section{Graphical illustrations of selected cases}

The formulas (\ref{dmeijerth}) and (\ref{dmeislaterformula}) alow us to study
the graphical representation of the function $V_{p,r}(x)$ for given $p=k/l>1$ and $r\in\mathbb{R}$.
Figure~\ref{figura2} shows $V_{p,0}(x)$ for $p=i/2$, $i=3,4,\ldots,7$.
Figures \ref{figura3}--\ref{figura5} illustrate $V_{3/2,r}$, $V_{5/3,r}$ and
$V_{7/2,r}$ for various choice of $r$, including $r>p-1$ when
$V_{p,r}(x)$ is negative for some $x$.
Those  $V_p,_r$'s which have negative parts are plotted with dashed lines.
Finally, in Figure~\ref{figura6}, we show graphs of $V_{p,-3/2}(x)$
for some values of $p$. Each of these functions is negative for some values of~$x$.


\begin{thebibliography}{99}



\bibitem{andrews}
G. E. Andrews, R. Askey, R. Roy,
\textit{Special Functions,}
Cambridge University Press, Cambridge 1999.

\bibitem{balnev}
N. Balakrishnan, V. B. Nevzorow,
\textit{A primer on statistical distributions,}
Wiley-Interscience, Hoboken, New Jersey 2003.






\bibitem{gkp}
R.~L.~Graham, D.~E.~Knuth, O.~Patashnik,
\textit{Concrete Mathematics. A Foundation for Computer Science,}
Addison-Wesley, New York 1994.





\bibitem{marichev}
O. I. Marichev, \textit{Handbook of Integral Transforms of Higher
Transcendental Functions-Theory and Algorithmic Tables,}
Ellis Horwood Ltd, Chichester, 1983.

\bibitem{mlotkowski2010}
W. M\l otkowski,
\textit{Fuss-Catalan numbers in noncommutative probability,}
Documenta Math. \textbf{15} (2010) 939--955.

\bibitem{mpz2012}
W. M{\l}otkowski, K. A. Penson, K. \.{Z}yczkowski,
\textit{Densities of the Raney distributions,}
 arXiv:1211.7259.

\bibitem{muraki}
N. Muraki,
\textit{Monotonic independence, monotonic central limit theorem
and monotonic law of small numbers,}
Inf. Dim. Anal. Quantum Probab. Rel. Topics \textbf{4} (2001) 39--58.

\bibitem{ns}
A.~Nica, R.~Speicher,
\textit{Lectures on the Combinatorics of Free Probability},
Cambridge University Press, 2006.

\bibitem{olver}
F. W. J. Olver, D. W. Lozier, R. F. Boisvert, C. W. Clark,
\textit{NIST Handbook of Mathematical Functions,}
Cambridge University Press, Cambridge 2010.


\bibitem{pezy}
K.~A.~Penson, K.~\.{Z}yczkowski,
\textit{Product of Ginibre matrices: Fuss-Catalan and Raney distributions}
Phys. Rev. E~\textbf{83} (2011) 061118, 9~pp.

\bibitem{polyanin}
A. D. Polyanin, A. V. Manzhirov,
\textit{Handbook of Integral Equations}, CRC Press, Boca Raton, 1998.

\bibitem{prudnikov3}
A. P. Prudnikov, Yu. A. Brychkov, O. I. Marichev,
\textit{Integrals and Series,}
Gordon and Breach, Amsterdam (1998)
Vol.~3: More special functions.

\bibitem{oeis}
N.~J.~A.~Sloane,
\textit{The On-line Encyclopedia of Integer Sequences,} (2013),
published electronically at: http://oeis.org/.

\bibitem{sneddon}
I. N. Sneddon,
\textit{The use of integral transforms,}
Tata Mac Graw-Hill Publishing Company, 1974.



\bibitem{vdn}
D.~V.~Voiculescu, K.~J.~Dykema, A.~Nica,
\textit{Free random variables}, CRM, Montr\'{e}al, 1992.


\bibitem{zpnc2011}
K. \.{Z}yczkowski, K. A. Penson, I. Nechita, B. Collins,
\textit{Generating random density matrices,}
J.~Math. Phys. \textbf{52} (2011) 062201, 20~pp.
\end{thebibliography}
\end{document}